\documentclass{daj}

\dajAUTHORdetails{%
  title = {Quantitative Invertibility of Random Matrices: a Combinatorial Perspective}, %% please capitalize all significant words
  author = {Vishesh Jain},
  plaintextauthor = {Vishesh Jain},
  keywords = {smoothed analysis, least singular value, inverse Littlewood--Offord problem.},
}   %%% END \dajAUTHORdetails

\dajEDITORdetails{%
   year={2021},
   number={10},
   received={9 July 2020},   % received date: example: 7 January 2017
   published={6 September 2021},  % published date
   doi={10.19086/da.27994},       % XXX = number of paper, e.g. da006 for paper#6
}   %%% END \dajEDITORdetails

 \usepackage{dsfont}
\usepackage{amsmath, amssymb, amsthm, verbatim,enumerate,bbm}
\usepackage[nameinlink,capitalise,noabbrev]{cleveref}
\crefname{appsec}{Appendix}{Appendices}
\creflabelformat{enumi}{#2textup{#1}#3}
\usepackage{bm}

\newtheorem{theorem}{Theorem}[section]

\newtheorem*{namedtheorem}{\theoremname}
\newcommand{\theoremname}{testing}

\newtheorem{lemma}[theorem]{Lemma}

\newtheorem{claim}[theorem]{Claim}

\newtheorem{proposition}[theorem]{Proposition}

\newtheorem{corollary}[theorem]{Corollary}

\newtheorem*{question*}{Question}

\theoremstyle{definition}
\newtheorem{definition}[theorem]{Definition}

\newtheorem{remark}[theorem]{Remark}

\theoremstyle{plain}

\DeclareMathOperator{\dist}{\text{dist}}

\begin{document}

\global\long\def\R{\mathbb{R}}

\global\long\def\S{\mathbb{S}}

\global\long\def\Z{\mathbb{Z}}

\global\long\def\C{\mathbb{C}}

\global\long\def\Q{\mathbb{Q}}

\global\long\def\N{\mathbb{N}}

\global\long\def\P{\mathbb{P}}

\global\long\def\F{\mathbb{F}}

\global\long\def\U{\mathcal{U}}

\global\long\def\V{\mathcal{V}}

\global\long\def\E{\mathbb{E}}

\global\long\def\Ev{\mathscr{Rk}}

\global\long\def\Dg{\mathscr{D}}

\global\long\def\Ndg{\mathscr{ND}}

\global\long\def\Rv{\mathcal{R}}

\global\long\def\Gv{\mathscr{Null}}

\global\long\def\Hv{\mathscr{Orth}}

\global\long\def\Supp{{\bf Supp}}

\global\long\def\Sv{\mathscr{Spt}}

\global\long\def\ring{\mathfrak{R}}

\global\long\def\1{\mathbbm{1}}

\global\long\def\Bad{{\bm{B}}}

\global\long\def\supp{{\bf supp}}

\global\long\def\A{\mathcal{A}}

\global\long\def\L{\mathcal{L}}

%\global\long\def\dist{\text{dist}}

\begin{frontmatter}[classification=text]
%% EDITOR: this will force the keywords to appear right after the Abstract.
%%   If the abstract is too long and would force the keywords off the
%%   front page, please comment out % [classification=text] above
%%   This way the keywords will be floated on the bottom of the first page
%%   even though the Abstract spills over to the next page.

%%% AUTHOR: Title goes here.  This line is optional.  You must use it
%%   if title has footnote attached or requires nontrivial typesetting,
%%   e.g., inclusion of linebreaks to force nice layout.
\title{Quantitative Invertibility of Random Matrices: a Combinatorial Perspective} %% please capitalize all significant words

%%% AUTHOR:
%%% List all authors. If you wish, place grant acknowledgements in \thanks.
%%% In brackets include a short tag for each author.
\author[vj]{Vishesh Jain}

%%% AUTHOR: Abstract goes here
\begin{abstract}
We study the lower tail behavior of the least singular value of an $n\times n$ random matrix $M_n := M+N_n$, where $M$ is a fixed complex matrix with operator norm at most $\exp(n^{c})$ and $N_n$ is a random matrix, each of whose entries is an independent copy of a complex random variable with mean $0$ and variance $1$. Motivated by applications, our focus is on obtaining bounds which hold with extremely high probability, rather than on the least singular value of a typical such matrix. 

This setting has previously been considered in a series of influential works by Tao and Vu, most notably in connection with the strong circular law, and the smoothed analysis of the condition number, and our results improve upon theirs in two ways: 
\begin{itemize}
\item We are able to handle $\|M\| = O(\exp(n^{c}))$, whereas the results of Tao and Vu are applicable only for $M = O(\text{poly(n)})$. 
\item Even for $M = O(\text{poly(n)})$, we are able to extract more refined information -- for instance, our results show that for such $M$, the probability that $M_n$ is singular is $O(\exp(-n^{c}))$, whereas even in the case when $\xi$ is a Bernoulli random variable, the results of Tao and Vu only give a bound of the form $O_{C}(n^{-C})$ for any constant $C>0$.    
\end{itemize}
As opposed to all previous works obtaining such bounds with error rate better than $n^{-1}$, our proof makes no use either of the inverse Littlewood--Offord theorems, or of any sophisticated net constructions. Instead, we show how to reduce the problem from the (complex) sphere to (Gaussian) integer vectors, where it is solved directly by utilizing and extending a combinatorial approach to the singularity problem for random discrete matrices, recently developed by Ferber, Luh, Samotij, and the author.

In particular, during the course of our proof, we extend the solution of the so-called `counting problem in inverse Littlewood-Offord theory' from Rademacher variables (established in the aforementioned work of Ferber, Luh, Samotij, and the author) to general complex random variables. This significantly improves on estimates for this problem obtained using the optimal inverse Littlewood-Offord theorem of Nguyen and Vu, and may be of independent interest.
\end{abstract}
\end{frontmatter}

\section{Introduction}
Let $M_n$ be an $n\times n$ complex matrix. Its \emph{singular values}, denoted by $s_k(M_n)$ for $k\in [n]$, are the eigenvalues of $\sqrt{M_{n}^{\dagger}M_{n}}$ arranged in non-increasing order. Of particular interest are the largest and smallest singular values, which admit the following variational characterizations:
$$s_1(M_{n}):= \sup_{\bm{x}\in \S^{2n-1}}\|M_{n}\bm{x}\|_{2};$$
$$s_n(M_{n}):= \inf_{\bm{x}\in \S^{2n-1}}\|M_{n}\bm{x}\|_{2},$$
where $\|\cdot \|_{2}$ denotes the usual Euclidean norm on $\C^{n}$, and $\S^{2n-1}$ denotes the set of unit vectors in $\C^{n}$. In this paper, we will be concerned with the following problem: for an $n\times n$ random matrix $M_n$ and a non-negative real number $\eta$, bound the probability $\Pr(s_n(M_n) \leq \eta)$ from above. This general problem captures, as special cases, many interesting and well-studied problems.

At one extreme, when $\eta=0$, the problem asks for an upper bound on the probability that $M_n$ is singular. Even in the case when the entries of $M_n$ are independent copies of a Rademacher random variable (i.e. a random variable which takes on the values $\pm 1$ with probability $1/2$ each), this is highly non-trivial. Considering the event that two rows or two columns of $M_n$ are equal (up to a sign) shows that $$\Pr(s_n(M_n) = 0) \geq (1+o_n(1))n^{2}2^{1-n},$$ and it has been conjectured since the 1950s that this  lower bound is tight. Despite this, even showing that $\Pr(s_n(M_n) = 0) = o_n(1)$ was only accomplished in 1967 by Koml\'os \cite{komlos1967determinant}, who used the Erd\H{o}s-Littlewood-Offord anti-concentration inequality to show that $\Pr(s_n(M_n) = 0) \lesssim n^{-1/2}$. 

A bound of the form 
$$\Pr(s_n(M_n) = 0) \leq (c + o_n(1))^{n},$$ for some $c \in (0,1)$, was obtained much later in 1995 by Kahn, Koml\'os, and Szemer\'edi \cite{kahn1995probability}, who proved such an estimate with $c = 0.999$. Subsequently, using deep ideas from additive combinatorics, Tao and Vu \cite{tao2007singularity} obtained such an estimate with $c = 0.75$, and by refining their ideas, Bourgain, Vu, and Wood \cite{bourgain2010singularity} were able to lower this constant to $c = 1/\sqrt{2}$. Recently, in a breakthrough work, Tikhomirov \cite{tikhomirov2018singularity} (building on the geometric approach to non-asymptotic random matrix theory pioneered by Rudelson and Vershynin \cite{rudelson2008littlewood}) showed that $\Pr(s_n(M_n) = 0) \leq (1/2 + o_n(1))^{n}$, thereby settling the singularity conjecture for random Rademacher matrices up to lower order terms.

At the other extreme, one may ask for the order of $s_n(M_n)$ for a `typical' realization of $M_n$; in our setup, this corresponds to the largest value of $\eta$ for which one can obtain a bound of the form $\Pr(s_n(M_n) \leq \eta ) \leq 0.01$ (say). For instance, confirming (in a very strong form) a conjecture of Smale, and a speculation of von Neumann and Goldstine, Edelman \cite{edelman1988eigenvalues} showed that for $M_n$ whose entries are independent copies of the standard Gaussian, 
$$\Pr(s_n(M_n) \leq \eta) \leq \sqrt{n}\eta;$$
this implies, in particular, that for i.i.d. standard Gaussian random matrices, $s_n(M_n)$ is typically $\Omega(n^{-1/2})$. Edelman's proof relied on special properties of the Gaussian distribution -- for general distributions, especially those which are allowed to have atoms, this question is much more challenging. 

In this case, building on intermediate work by Rudelson \cite{rudelson2008invertibility}, and essentially confirming a conjecture of Spielman and Teng, Rudelson and Vershynin \cite{rudelson2008littlewood} showed in a landmark work that for a real random matrix $M_n$ with i.i.d. centered subgaussian entries of variance $1$,
$$\Pr\left(s_n(M_n) \leq \eta \right) \lesssim \sqrt{n}\eta + c^{n},$$
which is optimal up to the constant $c\in (0,1)$ and the overall implicit constant. In recent years, much work has gone into establishing similar tail bounds under weaker assumptions: Rebrova and Tikhomirov \cite{rebrova2018coverings} established the same estimate as Rudelson and Vershynin for i.i.d centered random variables of variance $1$ (in particular, not assuming the existence of any moments higher than the second moment), and very recently (in fact, after the first version of the current paper appeared on the arXiv), Livshyts, Tikhomirov, and Vershynin \cite{livshyts2019smallest} obtained such an estimate for real random matrices $M_n$ whose entries are independent random variables satisfying a uniform anti-concentration estimate, and such that the expected sum of the squares of the entries is $O(n^2)$. Both of these works build upon the geometric framework of Rudelson and Vershynin.

For many applications, one would like to study random matrices whose entries have non-zero means. Whereas the results mentioned in the previous paragraph allow non-centered entries to some extent, they are unable to handle means larger than some threshold, due to their reliance on controlling various norms of the matrix. For instance, even the case when the mean of every entry is allowed to be in $[-n,n]$ has thus far remained out of reach of the geometric methods. Hence, the geometric methods fail to provide sufficiently powerful bounds in the important setting of \emph{smoothed analysis}, which we now discuss. 

\subsection{Smoothed analysis of the least singular value}
In their work on the \emph{smoothed analysis of algorithms} \cite{spielman2009smoothed, spielman2004smoothed} in numerical linear algebra, Spielman and Teng considered random matrices of the form
$M_n:= M + N_n$, where $M$ is a fixed (possibly `large') complex matrix, and $N_n$ is a complex random matrix with i.i.d. (centered) entries of variance $1$.  %independent and identically distributed (i.i.d.) entries. This problem has been intensely studied in the last couple of decades, see e.g. \cite{rudelson2010non, vershynin2010introduction, tao2008random, livshyts2019smallest} and the references therein.
Their motivation for studying this distribution on matrices was based on the following insight -- even if the desired input to an algorithmic problem is a fixed matrix $M$, it is likely that a computer will actually work with a perturbation $M+N_n$, where $N_n$ is a random matrix representing the effect of `noise' in the system. Sankar, Spielman, and Teng \cite{sankar2006smoothed} dealt with the case when the noise matrix $N_n$ has i.i.d. standard Gaussian entries, and found that such noise has a regularizing effect i.e. with high probability, the least singular value of $M_n$ is sufficiently large, even if this is not the case for $M$ itself. More precisely, they showed that for an arbitrary $n\times n$ matrix $M$, 
$$
\Pr\left(s_n(M_n) \leq \eta \right) \leq 2.35\sqrt{n}\eta,
$$
which is optimal up to the constant $2.35$. The proof of Sankar, Spielman, and Teng relied on special properties of the Gaussian distribution. Recently, using significantly different techniques, Tikhomirov \cite{tikhomirov2020invertibility} obtained such a result for all $N_{n}$ with  independent rows satisfying a technical assumption (this assumption is general enough to include isotropic log-concave distributions).

Motivated by more realistic noise models, especially those in which the noise distribution is allowed to have atoms (for instance, this is always the case with computers, see also the discussion in \cite{tao2010smooth}), Tao and Vu \cite{taocondition, tao2010smooth} investigated the lower tail behavior of $s_n(M_n)$ for very general noise matrices $N_n$. Using the so-called inverse Littlewood-Offord theory from additive combinatorics (see the discussion in \cref{sec:counting-problem-ILO}), they showed that for any complex random variable $\xi$ with mean $0$ and variance $1$, and for any constants $A,C > 0$, there exists a constant $B>0$ (depending on $A, C, \xi$; this is in general necessary, see \cite[Theorem~3.1]{tao2010smooth}) such that for any complex matrix $M$ with $\|M\|:= s_1(M) \leq n^{C}$, if $N_n$ is a complex random matrix whose entries are i.i.d. copies of $\xi$, then
\begin{equation}
\label{eqn:TV}    \Pr\left(s_n(M_n) \leq n^{-B}\right) \leq n^{-A}.
\end{equation}
Explicit dependence of $B$ on $A,C, \xi$ was given in \cite{tao2008random} and subsequently sharpened (but not optimally) in \cite{tao2010smooth}, although, for known applications of \cref{eqn:TV} in the literature, the exact dependence of $B$ on $A,C,\xi$ is not important for the analysis to go through (see the discussion in \cite{tao2010smooth}). 

However, in applications, it \emph{is} crucial that one can allow $A$ to be any positive constant -- this allows one to obtain estimates on $s_n(M_n)$ which can survive even a polynomial-sized (in $n$) union bound. As an example, in Tao and Vu's celebrated proof of the \emph{strong circular law} \cite{tao2008random, tao2010random}, it is essential to have an estimate of the form \cref{eqn:TV} for some $A > 1$. Proving estimates of the form \cref{eqn:TV} with $A > 1$ is significantly more involved than proving such estimates for some $A > 0$, and involves a much deeper understanding of the anti-concentration properties of vectors -- in particular, a decomposition of the sphere into just `compressible' and `incompressible' vectors, as is done in \cite{rudelson2008invertibility, gotze2010circular}, is insufficient for this purpose.

We also emphasize that the estimate in \cref{eqn:TV} holds for \emph{any complex random variable} with mean $0$ and variance $1$. Working with complex random variables of this generality provides significant additional challenges for the geometric methods, owing to the fact that the metric entropy of the unit sphere in $\C^{n}$ is twice that of the unit sphere in $\R^{n}$ (see the discussion in \cite{rudelson2016no}). Consequently, works based on the geometric method have thus far imposed further conditions on the dependence between the real and imaginary parts of the complex random variable, most commonly requiring the real and imaginary parts to be independent (see, e.g. \cite{rudelson2016no, luh2018complex}) in order to deduce bounds comparable to \cref{eqn:TV}. 

%Apart from the application to smoothed analysis, \cref{eqn:TV} (with $C=1/2$) also forms a crucial ingredient in Tao and Vu's celebrated proof of the  \emph{strong circular law} \cite{tao2008random, tao2010random}, which asserts that almost surely, as $n\to \infty$, the empirical spectral distribution of $N_n/\sqrt{n}$ converges to the uniform distribution over the unit disc in the complex plane; indeed, the bulk of \cite{tao2008random} is devoted to proving \cref{eqn:TV}.   

\subsection{Our results}
We introduce a new framework for providing estimates on the lower tail of $s_n(M_n)$ in the general setting of smoothed analysis, with a particular focus on values of $\eta$ `close' to $0$ (as opposed to obtaining the correct order of magnitude for `99 percent' of such matrices) . Our approach differs both from the geometric methods of Rudelson and Vershynin, as well as the additive combinatorial methods of Tao and Vu. Before discussing this further, we record our main result.

\begin{theorem}
\label{thm:main-smoothed-analysis}
Let $\xi$ be an arbitrary complex random variable with mean $0$ and variance $1$. 
Let $M$ be an $n\times n$ complex matrix with $\|M\| \leq 2^{n^{0.001}}$ and let $M_n = M + N_n$, where $N_n$ is a random matrix, each of whose entries is an independent copy of $\xi$. 

Then, for all $\alpha \geq 2^{-n^{0.001}}$ and for all $\eta \leq (C_{\ref{thm:main-smoothed-analysis}}(\|M\|+\sqrt{n})\alpha^{-1}n^{2})^{-300\log(\alpha^{-1})/\log{n}}$,
$$\Pr\left(s_n(M_n) \leq \eta \right) \leq C_{\ref{thm:main-smoothed-analysis}}\alpha,$$
where $C_{\ref{thm:main-smoothed-analysis}}$ is a constant depending only on $\xi$. 
\end{theorem}

%Recently, the author has put forward a novel combinatorial approach to non-asymptotic random matrix theory \cite{jain2019combinatorial, jain2019b} which builds upon and extends previous work of the author along with Ferber, Luh, and Samotij \cite{FJLS2018}. The main application in \cite{jain2019b} shows that if $M$ is an $n\times n$ complex matrix with $\|M\| = O(n^{(3/4) - \epsilon})$
%and $N_n$ is an $n\times n$ matrix with i.i.d. entries, each of which is a copy of a fixed complex random variable $\xi$ with mean $0$ and variance $1$, then
%$$\Pr\left(s_n(M_n) \leq \eta \right) \lesssim n^{5/2}\eta + \exp(-n^{c}).$$
%In particular, this provided a new and simple proof of the non-asymptotic part of the strong circular law, and to our knowledge, is the only result in this setting with an exponential-type error rate. However, there is a barrier to extending the methods of that work in case $\|M\| = \omega(n^{3/4})$, and therefore, it remained open whether: (1) one can obtain a non-asymptotic bound with exponential-type rate in the smoothed analysis setting, and (2) whether one can even obtain non-asymptotic bounds with polynomial error rate in the smoothed analysis setting \emph{without} using the inverse Littlewood--Offord theorems (we believe this latter question is interesting since \cref{eqn:TV} may perhaps be regarded as the principal application of inverse Littlewood--Offord theory to the study of non-Hermitian random matrices). \\

%In this paper, we answer both of these questions in the affirmative. Specifically, we show:

\begin{remark}
(1) When $\|M\| \geq \sqrt{n}$, the conclusion of \cref{thm:main-smoothed-analysis} shows that for all $t \in (0,1)$,
\[\Pr(s_n(M_n)\leq t) \leq C\max\left\{t^{\frac{c\log{n}}{\log{\|M\|}}} , n^{-c\sqrt{-\frac{\log{t}}{\log{n}}}}\right\} + C\cdot 2^{-n^{c'}},\]
where $c'$ is an absolute constant and $C,c$ are constants possibly depending on $\xi$.

(2) The choice of the upper bound $2^{n^{0.001}}$ on $\|M\|$ and $\alpha^{-1}$ is arbitrary and can certainly be improved, although we have made no attempt to do so.

(3) When $\alpha = n^{-A}$ and $\|M\| \leq n^{C}$, \cref{thm:main-smoothed-analysis} shows that $\Pr(s_n(M_n) \leq n^{-B}) = O(n^{-A})$ for some $B$ depending on $A$ and $C$, thereby recovering the result of Tao and Vu (up to the specific dependence of $B$ on $A$ and $C$, which, as noted earlier, is typically not important for applications). 
\end{remark}

\noindent \textbf{Discussion: }The main advantage of \cref{thm:main-smoothed-analysis} over \cref{eqn:TV} is that it is valid for $\alpha^{-1}, \|M\| \leq 2^{n^{0.001}}$, whereas \cref{eqn:TV} (recast in the form of \cref{thm:main-smoothed-analysis}) would provide a similar conclusion only for $\alpha^{-1}, \|M\| \leq O(\text{poly}(n))$. In particular, even in the case when $\|M\|$ is polynomially bounded in $n$ and $\xi$ is a Rademacher random variable, \cref{thm:main-smoothed-analysis} shows that $M_n$ is singular with probability at most $2^{-n^{0.001}}$, as compared to \cref{eqn:TV}, which only gives an inverse polynomial bound.

As mentioned earlier, our goal is to provide bounds in which one can take $\alpha$ to be very small (for instance, this is the case of interest in the singularity problem), and not so much on the exact relationship between $\eta$ and $\alpha, \|M\|$. However, we note that the main source of degradation in the relationship between $\eta$ and $\alpha, \|M\|$ in \cref{thm:main-smoothed-analysis} comes from a pigeonholing argument, introduced in \cite{tao2008random}. In \cite{tao2010smooth}, a better relationship between $\eta$ and $\alpha, \|M\|$ is obtained using a more involved pigeonholing scheme. By using this more involved scheme, the relationship between $\eta$ and $\alpha,\|M\|$ in \cref{thm:main-smoothed-analysis} can be made comparable to the current best known one in \cite{tao2010smooth}, although we have not attempted to do so in order to keep the exposition simple and transparent.

While \cref{thm:main-smoothed-analysis} significantly increases the range of validity of estimates like \cref{eqn:TV}, we feel that what is of greater interest are the proof techniques. Unlike the geometric methods, we make no use of net arguments (except very superficially). We also do not make any use of the inverse Littlewood--Offord theory of Tao and Vu. Instead, we utilize and extend an elementary combinatorial approach to the so-called `counting problem in inverse Littlewood--Offord theory' (see the next subsection), recently developed by Ferber, Luh, Samotij, and the author \cite{FJLS2018} -- this part of our paper may be of independent interest.

The benefit of this combinatorial approach to the counting problem is that it provides much better estimates than those that can be obtained from the inverse Littlewood--Offord theorems of Tao and Vu \cite{tao2010sharp}, and Nguyen and Vu \cite{nguyen2011optimal} -- this is, in part, because our approach is not hampered by the black-box application of heavy machinery from additive combinatorics. However, in contrast to the `continuous inverse Littlewood-Offord theorems' (\cite{tao2008random, nguyen2011optimal}), we do not have a genuinely `continuous version' of our counting results. This necessitates the need for additional arguments to reduce the quantitative invertibility problem to a situation where the `discrete counting theorem' we do have may directly be applied. Such an argument first appears in \cite{jain2019combinatorial}, where the author was able to use certain `rounding' arguments to avoid the need for a continuous version of the counting theorem; however, these arguments still relied on various norms of the random matrix not being too large, which is not true in the setting of smoothed analysis. Hence, the main technical challenge in the present work is to execute a version of these rounding arguments, even in the presence of large norms and heavy-tailed random variables.

At a high level, our work shows that for the purpose of controlling the smallest singular value of a random matrix, \emph{even in the general setting of smoothed analysis}, a good solution to the \emph{discrete counting version} of the inverse Littlewood--Offord problem (which, as we will see, is significantly easier to establish) is already sufficient. Note that a quantitatively weaker solution to this problem first appeared in the original breakthrough work of Tao and Vu on inverse Littlewood--Offord theory \cite{tao2009inverse}. However, in that work, the authors made use not just of the counting estimate, but also of the additive combinatorial structural information coming from the inverse Littlewood--Offord theorems in order to study the smallest singular value.

%We note that the proof of \cref{thm:main-smoothed-analysis} shows that one only needs to assume that (most of) the entries within each row are identically distributed. However, we do not yet know how to remove the identical distribution assumption, which we believe is just an artefact of the proof. It would be very interesting to see if one can exploit some of the tools and insights in the very recent work \cite{livshyts2019smallest} to extend our work to the case of independent but non-identically distributed entries.

\subsection{The counting problem in inverse Littlewood-Offord theory}
\label{sec:counting-problem-ILO}
In its simplest form, the so-called Littlewood-Offord problem, first raised by Littlewood and Offord in \cite{littlewood1943number} asks the following question. Let $\bm{a}:= (a_1,\dotsc, a_n) \in (\Z \setminus \{0\})^{n}$ and let $\epsilon_1,\dotsc, \epsilon_n$ be  i.i.d. Rademacher random variables. Estimate the largest atom probability $\rho(\bm{a})$, which is defined by
\[
  \rho(\bm{a}) := {\textstyle \sup_{x\in \Z}}\Pr\left(\epsilon_1 a_1 + \dotsb + \epsilon_n a_n = x\right).
\]
Littlewood and Offord showed that $\rho(\bm{a}) = O\left(n^{-1/2}\log{n}\right)$. Soon after, Erd\H{o}s~\cite{erdos1945lemma} gave an elegant combinatorial proof of the refinement $\rho(\bm{a}) \leq \binom{n}{\lfloor n/2 \rfloor} / 2^{n} = O(n^{-1/2})$, which is tight, as is readily seen by taking $\bm{a}$ to be the all ones vector. These classic results of Littlewood-Offord and Erd\H{o}s generated a lot of activity around this problem in various directions: higher-dimensional generalizations e.g. \cite{katona1966conjecture, kleitman1966combinatorial}; better upper bounds on $\rho(\bm{a})$ given additional hypotheses on $\bm{a}$ e.g. \cite{erdos1947e736, halasz1977estimates, sarkozy1965problem}; and obtaining similar results with the Rademacher distribution replaced by more general distributions e.g. \cite{esseen1966kolmogorov, halasz1977estimates}.   

A new view was brought to the Littlewood-Offord problem by Tao and Vu \cite{tao2009inverse, tao2008random} who, guided by inverse theorems from additive combinatorics, tried to find the underlying reason why $\rho(\bm{a})$ could be large. They used deep Freiman-type results from additive combinatorics to show that, roughly speaking, the only reason for a vector $\bm{a}$ to have $\rho(\bm{a})$ only polynomially small is that most coordinates of $\bm{a}$ belong to a generalized arithmetic progression (GAP) of `small rank' and `small volume'. Their results were subsequently sharpened by Nguyen and Vu~\cite{nguyen2011optimal}, who proved an `optimal inverse Littlewood--Offord theorem'. We refer the reader to the survey \cite{nguyen2013small} and the textbook \cite{tao2006additive} for complete definitions and statements, and much more on both forward and inverse Littlewood-Offord theory.
 
Recently, motivated by applications, especially those in random matrix theory such as the ones considered in the present work, the following \emph{counting variant} of the inverse Littlewood--Offord problem was isolated in work \cite{FJLS2018} of Ferber, Luh, Samotij, and the author: for \emph{how many} vectors $\bm{a}$ in a given collection $\mathcal{A}\subseteq \Z^{n}$ is the largest atom probability $\rho(\bm{a})$ greater than some prescribed value? The utility of such results is that they enable various union bound arguments, as one can control the number of terms in the relevant union/sum. %In fact, the inverse Littlewood-Offord theorems are typically used precisely through such counting corollaries \cite{nguyen2013small}. 
One of the main contributions of \cite{FJLS2018} was to show that one may obtain useful bounds for the counting variant of the inverse Littlewood-Offord problem directly, \emph{without} providing a precise structural characterization like Tao and Vu. Not only does this approach make certain arguments considerably simpler, it also provides better quantitative bounds for the counting problem, since it is not hampered by losses coming from the black-box application of various theorems from additive combinatorics. In \cite{FJLS2018, ferber2018singularity, jain2019combinatorial}, this work was utilized to provide quantitative improvements for several problems in combinatorial random matrix theory.

A natural question left open by this line of work is whether one can adapt the strategy of \cite{FJLS2018} to study the counting problem in inverse Littlewood-Offord theory with respect to general random variables as well. % random matrices whose entries have `continuous' distributions as well.
We note that the inverse Littlewood-Offord theorems in \cite{nguyen2011optimal, tao2008random} are indeed applicable to these more general settings. However, since the proofs in \cite{FJLS2018} proceed by viewing (bounded) integer-valued random variables as random variables valued in $\F_p$ (for sufficiently large $p$), it is not clear whether the combinatorial techniques there can be extended. Here, we show (\cref{thm:counting-continuous}), that the combinatorial arguments of \cite{FJLS2018} can be used in combination with (the dual of) the Fourier-analytic arguments in \cite{tao2008random, nguyen2011optimal} to prove a counting result  for very general distributions. The statement of the the following theorem uses \cref{defn:levy-conc} and \cref{defn:good-rv}.

\begin{theorem}
\label{thm:counting-continuous}
Let $\xi$ be a $C_\xi$-good random variable.  For $\rho \in (0,1)$ (possibly depending on $n$), let
$$\bm{V}_{\rho} :=\left\{\bm{v}\in (\Z+i\Z)^{n}: \rho_{1,\xi}(\bm{v}) \geq \rho\right\}.$$ There exists a constant $C_{\ref{thm:counting-continuous}} \geq 1$, depending only on $C_\xi$, for which the following holds. Let $n,s,k\in \N$ with $1000C_{\xi} \leq k\leq \sqrt{s}\leq s \leq n/\log{n}$. If $\rho \geq C_{\ref{thm:counting-continuous}}\max\left\{ e^{-s/k},s^{-k/4}\right\}$ and $p$ is an odd prime such that $2^{n/s}\geq p \geq C_{\ref{thm:counting-continuous}}\rho^{-1}$, then 
$$\left|\varphi_p(\bm{V}_\rho)\right| \leq \left(\frac{5np^{2}}{s}\right)^{s} + \left(\frac{C_{\ref{thm:counting-continuous}}\rho^{-1}}{\sqrt{s/k}}\right)^{n},$$
where $\varphi_p$ denotes the natural map from $(\Z+i\Z)^{n}\to (\F_p+i\F_p)^{n}$.
\end{theorem}

\begin{remark}
The inverse Littlewood-Offord theorems may be used to deduce similar statements, \emph{provided we further assume that $\rho \geq n^{-C}$ for some constant $C>0$}. The freedom of taking $\rho$ to be much smaller is the source of the quantitative improvements in \cref{thm:main-smoothed-analysis}.
\end{remark}

\noindent {\bf Organization: }The rest of this paper is organized as follows. In \cref{sec:preliminaries}, we collect some preliminary results on anti-concentration; the main result of this section is \cref{prop:refined-diophantine}. In \cref{sec:warmup}, as a warm-up (included in lieu of an informal sketch of the proof), we provide a proof of \cref{thm:main-smoothed-analysis} under the additional assumption that the random variable $\xi$ is subgaussian. In \cref{sec:proof-main}, we provide a proof of \cref{thm:main-smoothed-analysis}; this follows essentially the same outline as in the subgaussian case, with the main difference being \cref{prop:counting} (and the supporting results required to prove it). Finally, in \cref{sec:proof-counting-continuous}, we prove \cref{thm:counting-continuous}. \\

\noindent {\bf Notation: } Throughout the paper, we will omit floors and ceilings when they make no essential difference. For convenience, we will also say `let $p = x$ be a prime', to mean that $p$ is a prime between $x$ and $2x$; again, this makes no difference to our arguments. We will use $\S^{2n-1}$ to denote the set of unit vectors in $\C^{n}$, $B(x,r)$ to denote the ball of radius $r$ centered at $x$, and $\Re(\bm{v}), \Im(\bm{v})$ to denote the real and imaginary parts of a complex vector $\bm{v}\in \C^{n}$. As is standard, we will use $[n]$ to denote the discrete interval $\{1,\dots,n\}$. We will also use the asymptotic notation $\lesssim, \gtrsim, \ll, \gg$ to denote $O(\cdot), \Omega(\cdot), o(\cdot), \omega(\cdot)$ respectively. For a matrix $M$, we will use $\|M\|$ to denote its standard $\ell^{2}\to \ell^{2}$ operator norm. All logarithms are natural unless noted otherwise.

\section{Preliminaries}
\label{sec:preliminaries}
In this section, we collect some tools and auxiliary results that will be used throughout the rest of this paper.
\begin{definition}[L\'evy concentration function]
\label{defn:levy-conc}
Let $n\in \mathbb{Z}^{\geq 1}$, let $\bm{\xi}:=(\xi_1,\dots,\xi_n) \in \C^{n}$ be a random vector, and let $\bm{v}:=(v_{1},\dots,v_{n})\in\C^{n}$. We define the \emph{L\'evy concentration function of $\bm{v}$ at radius $r \in \R^{\geq 0}$ with respect to $\bm{\xi}$} by $$\rho_{r,\bm{\xi}}(\bm{v}):=\sup_{x\in\C}\Pr\left(v_{1}\xi_{1}+\dots+v_{n}\xi_{n}\in B(x,r)\right).$$
\end{definition}
\begin{remark}
(1) For lightness of notation, we have chosen to omit the ambient dimension $n$ from $\rho_{r,\bm{\xi}}(\bm{v})$. This should not create any confusion since the dimension of $\bm{v}$ or $\bm{\xi}$ will always be clear from context. 

(2) Note that when $n=1$ and $\xi$ is a random variable taking values in $\C$, we have that $\rho_{r,\xi}(1) = \sup_{x\in \C}\Pr(\xi \in B(x,r))$. We will use this notation repeatedly. 

(3) Moreover, when the components of $\bm{\xi}$ are i.i.d.~copies of some random variable $\xi$, we will sometimes abuse notation by using $\rho_{r,\xi}(\bm{v})$ to denote $\rho_{r,\bm{\xi}}(\bm{v})$.

(4) If $\tilde{\bm{\xi}}$ is a random vector whose distribution coincides with that of a random vector $\bm{\xi}$ conditioned on some event $\mathcal{E}$, then we will  often denote $\rho_{r,\tilde{\bm{\xi}}}(\bm{v})$ by $\rho_{r,\bm{\xi}|\mathcal{E}}(\bm{v})$.   
\end{remark}

The next lemma shows that weighted sums of random variables which are not close to being a constant are also not close to being a constant. 
\begin{lemma}(see, e.g., Lemma 6.3 in \cite{tao2010smooth}) 
\label{lemma:anticoncentration}
Let $\xi$ be a complex random variable with finite non-zero variance. Then, there exists a constant $c_{\ref{lemma:anticoncentration}} \in (0,1)$, depending only on $\xi$, such that 
$$\sup_{\bm{v}\in \S^{2n-1}}\rho_{c_{\ref{lemma:anticoncentration}},\xi}(\bm{v}) \leq 1-c_{\ref{lemma:anticoncentration}}.$$
\end{lemma}

Combining this with the so-called tensorization lemma (see Lemma 2.2 in \cite{rudelson2008littlewood}), we get the following estimate for `invertibility with respect to a single vector'.

\begin{lemma}
\label{lemma:invertibility-single-vector}
Let $\xi$ be a complex random variable with finite non-zero variance.
Let $M$ be an arbitrary $n\times n$ matrix and let $N_n$ be a random matrix each of whose entries is an independent copy of $\xi$. Then, for any fixed $\bm{v}\in \S^{2n-1}$,
$$\Pr\left(\|(M+N_{n})\bm{v}\|_2\leq c_{\ref{lemma:invertibility-single-vector}}\sqrt{n}\right) \leq (1- c_{\ref{lemma:invertibility-single-vector}})^{n},$$
where $c_{\ref{lemma:invertibility-single-vector}}\in (0,1)$ is a constant depending only on $\xi$. 
\end{lemma}

We will also need the following simple fact, which compares the L\'evy concentration function with respect to a random vector to the L\'evy concentration function with respect to a conditioned version of the random vector. 
\begin{lemma}
\label{lemma:compare-sbp-condition}
Let $\bm{\xi}:=(\xi_1,\dots,\xi_n)$ be a complex random vector, let $\mathcal{G}$ be an event depending on $\bm{\xi}$, and let $\bm{\tilde{\xi}}$ denote a random vector distributed as $\bm{\xi}$ conditioned on $\mathcal{G}$. Then, for any $\bm{v} \in \C^{n}$ and for any $r\geq 0$,
$$\rho_{r,\bm{\xi}}(\bm{v}) \geq \rho_{r,\bm{\tilde{\xi}}}(\bm{v})\Pr(\mathcal{G}).$$
\end{lemma}
\begin{proof}
Fix $\epsilon > 0$ and let $x\in \C$ be such that
$$\Pr\left(v_1 \xi_1 + \dots + v_n \xi_n \in B(x,r) \mid \mathcal{G}\right) \geq \rho_{r,\bm{\tilde{\xi}}}(\bm{v}) - \epsilon.$$
Then, we have
\begin{align*}
    \Pr\left(v_1\xi_1+\dots+v_n \xi_n \in B(x,r)\right)
    &\geq \Pr\left(v_1\xi_1+\dots+v_n \xi_n \in B(x,r) \cap \mathcal{G}\right)\\
    &= \Pr\left(v_1\xi_1+\dots+v_n \xi_n \in B(x,r) \mid \mathcal{G}\right)\Pr(\mathcal{G})\\
    &\geq \rho_{r,\bm{\tilde{\xi}}}(\bm{v})\Pr(\mathcal{G}) - \epsilon.
\end{align*}
Taking the supremum of the left hand side over the choice of $x\in \C$, and then taking the limit of the right hand side as $\epsilon \to 0$ completes the proof.
\end{proof}

In order to state the main assertion of this subsection (\cref{prop:refined-diophantine}), we need the following definition. 

\begin{definition}
\label{defn:good-rv}
We say that a random variable $\xi$ is \emph{$C$-good} if 
\begin{equation}
\label{eqn:assumption-on-z}
\Pr(C^{-1}\leq|\xi_{1}-\xi_{2}|\leq C)\geq C^{-1},
\end{equation}
where $\xi_1$ and $\xi_2$ denote independent copies of $\xi$. 
The smallest $C\geq 1$ with respect to which $\xi$ is $C$-good will be denoted by $C_\xi$. 
\end{definition}

The following lemma shows that the general random variables with which we are concerned in this paper (i.e. complex random variables with finite non-zero variance) are indeed $C$-good for some finite $C$, so that there is no loss of generality for us in imposing this additional restriction. 

\begin{lemma}
\label{lemma:non-trivial-implies-good}
Let $\xi$ be a complex random variable with variance $1$. Then, $\xi$ is $C_{\xi}$-good for some $C_{\xi} \geq 1$. \end{lemma}
\begin{proof}
Since $\text{Var}(\xi)=1$, there must exist some $u_\xi, v_\xi \in (0,1)$ such that $\rho_{v_\xi,\xi}(1)\leq u_\xi$. Therefore, letting $\xi'$ denote an independent copy of $\xi$, we have
\begin{align*}
    \Pr\left(|\xi - \xi'|\leq\frac{v_\xi}{2}\right) \leq \rho_{v_\xi,\xi-\xi'}(1) \leq \rho_{v_\xi,\xi}(1) \leq u_\xi. 
\end{align*}
Moreover, since $\E[|\xi-\xi'|^{2}]=\text{Var}(\xi-\xi') = \text{Var}(\xi) + \text{Var}(\xi') = 2$, it follows from Markov's inequality that
$$\Pr\left(|\xi-\xi'| \geq 2(1-u_{\xi})^{-1/2}\right)\leq \frac{1-u_{\xi}}{2}.$$
Combining these two bounds, we see that
\begin{align*}
    \Pr\left(\frac{v_{\xi}}{2}\leq |\xi - \xi'| \leq 2(1-u_{\xi})^{-1/2}\right)\geq \frac{1-u_{\xi}}{2},
\end{align*}
which gives the desired conclusion.
\end{proof}

We conclude this subsection with the following proposition, which roughly states that the L\'evy concentration function of a vector with no suitable multiple sufficiently close to a Gaussian integer vector must be small. This will prove crucial in our replacement of applications of the continuous inverse Littlewood-Offord theorem by \cref{thm:counting-continuous}.

\begin{proposition}
\label{prop:refined-diophantine}
Let $\xi_{1},\dots,\xi_{n}$
be independent copies of a $C_{\xi}$-good complex random variable $\xi$. % satisfying $\rho_{v_\xi,\xi}(1)\leq u_\xi$ for some $v_\xi \in (0,1], u_\xi \in (0,1)$.  %$Q(X)\leq1-p$,
%where $Q(X):=\sup_{x\in\R}\Pr\{|X-x|\leq1\}$. 
Let $\bm{v}:=(v_{1},\dots,v_{n})\in\C^{n}\setminus\{\bm{0}\}$.
Suppose the following holds: there exists some $f(n) \in (0,1)$, $g(n) \in (1,\infty)$ and $\alpha > 0$ such that 
\begin{align*}
\dist(\eta \bm{v},(\Z+i\Z)^{n}) & \geq\alpha\quad \forall\eta\in \C \text{ such that } |\eta| \in\left[f(n),g(n)\right].
\end{align*}
Then, for any $r\geq 0$,
\[
\rho_{r,\xi}(\bm{v})\leq C_{\ref{prop:refined-diophantine}}\exp(\pi r^{2})\left(\exp\left(-c_{\ref{prop:refined-diophantine}}g(n)^{2}\right) + \exp\left(-c_{\ref{prop:refined-diophantine}}\alpha^{2}\right) + f(n) \right),
\]
where $C_{\ref{prop:refined-diophantine}}\geq 1$ and $c_{\ref{prop:refined-diophantine}}>0$ are constants depending only on $C_{\xi}$.
\end{proposition}

The proof of this proposition requires the following preliminary definition and short Fourier-analytic lemmas from \cite{tao2008random}, along with a `doubling trick' appearing in \cite{jain2019b}. 

\begin{definition}
Let $\xi$ be an arbitrary complex random variable. For any $w\in\C$, we define 
$$\|w\|_{\xi}^{2}:=\E\|\Re\{w(\xi_{1}-\xi_{2})\}\|_{\R/\Z}^{2},$$
where $\xi_{1},\xi_{2}$ denote i.i.d. copies of $z$ and $\|\cdot\|_{\R/\Z}$ denotes the distance to the nearest integer. 
\end{definition}

\begin{lemma}[Lemma 5.2 in \cite{tao2008random}]
\label{lemma:initial-fourier-bound}
Let
$\bm{v}:=(v_{1},\dots,v_{n})\in\C^{n}$ and let $\xi$ be an
arbitrary complex random variable. Then, 
$$\rho_{r,\xi}(\bm{v})\le e^{\pi r^{2}}P_{\xi}(\bm{v}) \leq  e^{\pi r^{2}}\int_{\C}\exp\left(-\sum_{i=1}^{n}\|v_{i}z\|_{\xi}^{2}/2-\pi|z|^{2}\right)dz.$$
Here, 
$$P_\xi(\bm{v}) := \E_{x_1,\dots,x_n}\exp(-\pi|v_1 x_1 + \dots + v_n x_n|^{2}),$$
where $x_1,\dots,x_n$ are i.i.d. copies of $(\xi_1 - \xi_2)\cdot \text{Ber} (1/2)$, with $\xi_1,\xi_2$  distributed as $\xi$, and  $\text{Ber}(1/2), \xi_1, \xi_2$ mutually independent.  
\end{lemma}

\begin{lemma}[Lemma 4.5 (iii) in \cite{tao2008random}]
\label{lemma:doubling}
For $\bm{v}, \bm{w}\in \C^{n}$, let $\bm{v}\bm{w} \in \C^{2n}$ denote the vector whose first $n$ coordinates coincide with $\bm{v}$ and last $n$ coordinates coincide with $\bm{w}$. Then,
$$P_{\xi}(\bm{v})P_{\xi}(\bm{w}) \leq 2P_{\xi}(\bm{v}\bm{w}).$$
\end{lemma}

\begin{proof}[Proof of \cref{prop:refined-diophantine}]
Let $\bm{w}\in \C^{2n}$ denote the vector whose first $n$ components are $\bm{v}$ and last $n$ components are $i\bm{v}$. Then, we have
  \begin{align*}
      \rho_{r,\xi}(\bm{v})^{2}
      &= \rho_{r,\xi}(\bm{v})\rho_{r,\xi}(i\bm{v})\\
      &\leq \exp(2\pi r^{2})P_{\xi}(\bm{v})P_{\xi}(i\bm{v})\\
      &\leq 2\exp(2\pi r^{2})P_{\xi}(\bm{w})\\
      &\leq 2\exp(2\pi r^{2})\int_{\C}\exp\left(-\sum_{j=1}^{n}\left(\|v_{j}z\|_{\xi}^{2} + \|i v_{j}z\|_{\xi}^{2}\right)/2 - \pi |z|^{2}\right)dz,
  \end{align*}
  where the first line uses $\rho_{r,\xi}(\bm{v}) = \rho_{r,\xi}(i\bm{v})$, the second line is due to \cref{lemma:initial-fourier-bound}, the third line follows from \cref{lemma:doubling}, and the last line is again due to \cref{lemma:initial-fourier-bound}.
  
  Next, note that
  \begin{align*}
      \sum_{j=1}^{n}\left(\|v_j z\|_{\xi}^{2} + \|i v_j z\|_{\xi}^{2}\right)&=
      \E\sum_{j=1}^{n}\left(\|\Re\{v_j z (\xi_1 - \xi_2)\}\|_{\R/\Z}^{2} + \|\Re\{iv_j z (\xi_1-\xi_2)\}\|_{\R/\Z}^{2}\right)\\
      &= \E\sum_{j=1}^{n}\left(\|\Re\{v_j z (\xi_1 - \xi_2)\}\|_{\R/\Z}^{2} + \|\Im\{v_j z (\xi_1-\xi_2)\}\|_{\R/\Z}^{2}\right)\\
      &= \E\left[ \dist^{2}\left(\bm{v} z (\xi_1-\xi_2), (\Z + i\Z)^{n}\right)\right]\\
      &\geq \E\left[ \dist^{2}\left(\bm{v} z (\xi_1-\xi_2), (\Z + i\Z)^{n}\right) \bigg\vert |\xi_1 - \xi_2| \in [C_\xi^{-1}, C_{\xi}]\right]C_{\xi}^{-1},
  \end{align*}
  where the final inequality follows from the $C_\xi$-goodness of $\xi$.
  Therefore, from Jensen's inequality, we get that
  \begin{align}
      \rho_{r,\xi}(\bm{v})^{2}&
      \leq 2e^{2\pi r^2}\E \left[\int_{\C}\exp(-C_{\xi}^{-1}\dist^{2}\left(\bm{v} z (\xi_1-\xi_2), (\Z + i\Z)^{n}\right)/2 - \pi |z|^{2})dz \bigg\vert |\xi_1 - \xi_2| \in [C_{\xi}^{-1}, C_{\xi}] \right] \nonumber \\
      &\leq 2\exp(2\pi r^{2} )\sup_{|y| \in [C_\xi^{-1}, C_{\xi}]}\int_{\C}\exp(-C_{\xi}^{-1}\dist^{2}\left(\bm{v} z y, (\Z + i\Z)^{n}\right)/2 - \pi |z|^{2})dz \nonumber \\
      &\leq 2\exp(2\pi r^{2} )\sup_{|y| \in [C_\xi^{-1}, C_{\xi}]}\int_{\C}\exp(-C_{\xi}^{-1}\dist^{2}\left(\bm{v} z, (\Z + i\Z)^{n}\right)/2 - \pi |z/y|^{2})\frac{dz}{y}. 
      \label{eqn:fourier-bound-appendix}
  \end{align}

Let $A_1:=\{z \in \C \mid \dist(z \bm{v},(\Z+i\Z)^{n})\geq\alpha\}$, let $A_2:= \{z \in \C \mid |z| \in [0,g(n)]\}\setminus A_1$, and let $A_3:= \{z \in \C: |z| \in (g(n),\infty)\} \setminus A_1$. 
Then, we can bound the integral on the right hand side in \cref{eqn:fourier-bound-appendix} from above by 
\[
\sup_{|y|\in [C_{\xi}^{-1}, C_{\xi}]}\int_{A_1}+\sup_{|y|\in [C_{\xi}^{-1}, C_{\xi}]}\int_{A_2} + \sup_{|y|\in [C_{\xi}^{-1}, C_{\xi}]}\int_{A_3}.
\]
Let us, in turn, bound each of these three terms separately. 
\begin{itemize}

\item For the first term, we have the estimate
\begin{align*}
\sup_{|y|\in[C_{\xi}^{-1},C_{\xi}]}\int_{A_{1}} & \leq\exp\left(-C_{\xi}^{-1}\alpha^{2}/2\right)\sup_{|y|\in[C_{\xi}^{-1},C_{\xi}]}\int_{\C}\exp\left(-\pi\frac{|z|^{2}}{|y|^{2}}\right)\frac{dz}{y}\\
 & \leq 100\exp\left(-C_{\xi}^{-1}\alpha^{2}/2\right).
\end{align*}

\item For the second term, we begin by noting that since $\{z\in \C \mid |z| \in [f(n), g(n)]\}\subseteq A_1$ by assumption, it follows that $A_2 = \{z\in \C \mid |z| \in [0,f(n)]\} \setminus A_1$. Therefore, we have the trivial estimate
\begin{align*}
    \sup_{|y|\in[C_{\xi}^{-1},C_{\xi}]}\int_{A_{2}} & \leq\sup_{|y|\in[C_{\xi}^{-1},C_{\xi}]}\int_{\C \cap B(0,f(n))}\exp\left(-\pi \frac{|z|^{2}}{|y|^{2}}\right)\frac{dz}{y}\\
 & \leq 10C_{\xi}^{2}f(n)^{2}.
\end{align*}

\item For the third term, we have the estimate
\begin{align*}
\sup_{|y|\in[C_{\xi}^{-1},C_{\xi}]}\int_{A_{3}} & \leq\sup_{|y|\in[C_{\xi}^{-1},C_{\xi}]}\int_{\C \setminus B(0,g(n))}\exp\left(-\pi\frac{|z|^{2}}{|y|^{2}}\right)\frac{dz}{y}\\
 &\leq100\exp\left(-\frac{C_{\xi}^{-2}g(n)^{2}}{20}\right).
\end{align*}

\end{itemize}
Finally, summing the estimates in the previous three bullet points and taking the square root gives the desired conclusion.
\end{proof}

\section{Warm-up: proof of \cref{thm:main-smoothed-analysis} in the subgaussian case}
\label{sec:warmup}
In this section, we will discuss the proof of \cref{thm:main-smoothed-analysis} in the special case when the entries are further assumed to be i.i.d. subgaussian. This will allow the reader to see many of the key ideas and calculations in a simpler, less technical, setting. Our general reduction and outline follows Tao and Vu \cite{tao2008random, tao2010smooth}; as mentioned in the introduction, the main difference is the replacement of the crucial continuous inverse Littlewood-Offord theorem. 
\begin{definition}
\label{defn:subgaussian}
A complex random variable $\xi$ is said to be $C$-subgaussian if, for all $t>0$,
$$\Pr\left(|\xi|>t\right) \leq 2\exp\left(-\frac{t^2}{C^2}\right).$$
\end{definition}

For the remainder of this section, we fix a centered $\tilde{C}_{\xi}$-subgaussian complex random variable $\xi$ with variance $1$. Our goal in this section is to prove the following subgaussian version of \cref{thm:main-smoothed-analysis}.

\begin{theorem}
\label{thm:main-subgaussian}
Let $\xi$ be a centered $\tilde{C}_{\xi}$-subgaussian complex random variable with variance $1$. Let $M$ be an $n\times n$ complex matrix with $\|M\| \leq 2^{n^{0.001}}$ and let $M_n = M + N_n$, where $N_n$ is a random matrix, each of whose entries is an independent copy of $\xi$. 

Then, for all $\alpha \geq 2^{-n^{0.001}}$ and for all $\eta \leq (C_{\ref{thm:main-subgaussian}}(\|M\|+\sqrt{n})\alpha^{-1}n^{2})^{-300\log(\alpha^{-1})/\log{n}}$,
$$\Pr\left(s_n(M_n) \leq \eta \right) \leq C_{\ref{thm:main-subgaussian}}\alpha,$$
where $C_{\ref{thm:main-subgaussian}}\geq 1$ is a constant depending only on $\xi$. 
\end{theorem}

\subsection{Properties of subgaussian random variables}

%Since the statement of \cref{thm:sing-heavy-tailed} is `invariant' under an overall rescaling of the matrix $M_n$ by a constant (such a rescaling only changes the constant implicit in `$\lesssim$'), we may (and will) henceforth assume that the entries of $M_n$ are $1$-subgaussian. \\ 

A basic and important fact about subgaussian random variables is the so-called subgaussian concentration inequality. 
\begin{lemma}[see, e.g., Proposition 5.10 in \cite{vershynin2010introduction}]
\label{lemma:subgaussian-concentration}
Let $\xi_1,\dots, \xi_n$ be independent centered $\tilde{C}_{\xi}$-subgaussian complex random variables. Then, for every $\bm{v}:=(v_1,\dots,v_n) \in \C^{n}$ and for every $t\geq 0$, we have
$$\Pr\left(\left|\sum_{i=1}^{n}v_i \xi_i\right| \geq t\right) \leq 3\exp\left(-\frac{c_{\ref{lemma:subgaussian-concentration}}t^2}{\|\bm{v}\|_{2}^{2}}\right),$$
where $c_{\ref{lemma:subgaussian-concentration}} > 0$ is a constant depending only on $\tilde{C}_{\xi}$.
\end{lemma}

The subgaussian concentration inequality allows us to show that if $\bm{a}, \bm{b} \in \C^{n}$ are close in Euclidean distance, then the L\'evy concentration functions of $\bm{a}$ and $\bm{b}$ are close in a suitable sense as well. More precisely:

\begin{proposition}
\label{prop:approximation-sbp-subgaussian}
Let $\bm{\xi}:=(\xi_1,\dots,\xi_{n})$ be a complex random vector whose entries are independent centered $\tilde{C}_{\xi}$-subgaussian complex random variables. Then, for every $\bm{a}:=(a_1,\dots,a_n), \bm{b}:= (b_1,\dots,b_n) \in \C^{n}$, and for every $r_1,r_2\geq 0$, we have
$$\rho_{r_1+r_2, \bm{\xi}}(\bm{b}) \geq \rho_{r_1,\bm{\xi}}(\bm{a}) - 3\exp\left(-\frac{c_{\ref{lemma:subgaussian-concentration}}r_2^{2}}{\|\bm{a}-\bm{b}\|_{2}^{2}}\right).$$
\end{proposition}
\begin{proof}
For $r_2\geq 0$, let $\mathcal{E}_{r_2}$ denote the event that $\left|\sum_{i=1}^{n}(a_i - b_i)\xi_i\right| < r_2$. By \cref{lemma:subgaussian-concentration}, 
$$\Pr\left(\mathcal{E}_{r_2}^{c}\right) \leq 3\exp\left(-\frac{c_{\ref{lemma:subgaussian-concentration}}r_2^{2}}{\|\bm{a}-\bm{b}\|_{2}^{2}}\right).$$
Fix $\epsilon > 0$, and let $x \in \C$ be such that 
$$\Pr\left(a_1 \xi_1 + \dots + a_n \xi_n \in B(x,r_1)\right) \geq \rho_{r_1,\bm{\xi}}(\bm{a}) - \epsilon.$$
Then, 
\begin{align*}
    \Pr\left(b_1\xi_1 +\dots + b_n\xi_n \in B(x,r_1+r_2)\right)
    & \geq \Pr\left(b_1 \xi_1 + \dots + b_n \xi_n \in B(x,r_1+r_2) \cap \mathcal{E}_{r_2}\right)\\
    & \geq \Pr\left(a_1 \xi_1 + \dots + a_n \xi_n \in B(x,r_1) \cap \mathcal{E}_{r_2}\right)\\
    & \geq \Pr\left(a_1 \xi _1 + \dots + a_n \xi_n \in B(x,r_1)\right) - \Pr(\mathcal{E}_{r_2}^{c})\\
    & \geq \rho_{r_1,\bm{\xi}}{(\bm{a})}-\epsilon - \Pr(\mathcal{E}_{r_2}^{c}), 
\end{align*}
where the second line follows from the triangle inequality.

Taking the supremum of the left hand side over the choice of $x\in \C$, and then taking the limit on the right hand side as $\epsilon \to 0$ gives the desired conclusion. 
\end{proof}

\begin{remark}
As will be seen later, the key technical challenge in extending the proof of \cref{thm:main-smoothed-analysis} from the subgaussian case to the general case is the unavailability of \cref{prop:approximation-sbp-subgaussian}. 
\end{remark}

Finally, we need the following well-known estimate on the operator norm of a random matrix with i.i.d. subgaussian entries, which may be proved by combining the subgaussian concentration inequality with a standard epsilon-net argument. 

\begin{lemma}[see, e.g., Lemma 2.4 in \cite{rudelson2008littlewood}] 
\label{lemma:operator-norm-subgaussian}
Let $N_n$ be an $n\times n$ random matrix whose entries are i.i.d. centered $\tilde{C}_{\xi}$-subgaussian complex random variables. Then,
$$\Pr\left(\|N_{n}\| \geq C_{\ref{lemma:operator-norm-subgaussian}}\sqrt{n}\right) \leq 2\exp(-n),$$
where $C_{\ref{lemma:operator-norm-subgaussian}}\geq 1$ depends only on $\tilde{C}_{\xi}$.
\end{lemma}

\subsection{Rich and poor vectors}
For the remainder of this section, we fix an $n\times n$ complex matrix $M$ and parameters $\alpha,\eta \in (0,1)$ satisfying the restrictions of the statement of \cref{thm:main-subgaussian}. Also, let $$
\beta: = \frac{\alpha}{n},\quad f(\beta) := \frac{\beta}{100C_{\ref{prop:refined-diophantine}}} \in (0,1), \quad J(\beta, n):= \frac{100\log(\beta^{-1})}{\log{n}}.$$ We may assume without loss of generality that $\|M\| \geq 2C_{\ref{lemma:operator-norm-subgaussian}}\sqrt{n}$ as otherwise, an improved version of \cref{thm:main-smoothed-analysis} already follows from the main result in \cite{jain2019b}. We may also assume that $\eta \geq 2^{-n^{0.01}}$, since the statement of \cref{thm:main-subgaussian} for smaller values of $\eta$ follows from the result for $\eta = 2^{-n^{0.01}}$. 
Following Tao and Vu \cite{tao2008random}, we call a unit vector $\bm{v} \in \C^{n}$ \emph{poor} if we have
$$\rho_{2\eta\sqrt{n},\bm{\xi}}(\bm{v}) \leq \beta$$
and \emph{rich} otherwise.
We use $\bm{P}(\beta)$ and $\bm{R}(\beta)$ to denote, respectively, the set of poor and rich vectors. Accordingly, we have
\begin{align*}
    \Pr\left(s_n(M_n)\leq \eta\right) \leq \Pr\left(\exists \bm{v}\in \bm{P}(\beta): \|M_{n}\bm{v}\|_{2} \leq \eta\right) + \Pr\left(\exists \bm{v}\in \bm{R}(\beta): \|M_{n}\bm{v}\|_{2} \leq \eta\right).
\end{align*}
Therefore, \cref{thm:main-subgaussian} is a consequence of the following two propositions and the union bound. 

\begin{proposition}
\label{prop:eliminate-poor-subgaussian}
$\Pr\left(\exists \bm{v} \in \bm{P}(\beta): \|M_{n}\bm{v}\|_{2}\leq \eta \right) \leq n\beta$. 
\end{proposition}
\begin{proposition}
\label{prop:eliminate-rich-subgaussian}
$\Pr\left(\exists \bm{v} \in \bm{R}(\beta): \|M_{n}\bm{v}\|_{2}\leq \eta \right) \leq C_{\ref{prop:eliminate-rich-subgaussian}}\exp(-c_{\ref{prop:eliminate-rich-subgaussian}}n),$
where $C_{\ref{prop:eliminate-rich-subgaussian}}\geq 1$ and $c_{\ref{prop:eliminate-rich-subgaussian}}>0$ are constants depending only on $\xi$.
\end{proposition}

The proof of \cref{prop:eliminate-poor-subgaussian} is relatively simple, and follows from a conditioning argument developed in \cite{litvak2005smallest} (see, e.g., the proof of Lemma 11.3 in \cite{tao2008random}). We omit the details here, since later in \cref{prop:eliminate-poor}, we will prove a similar (but more complicated, and with a slightly different conclusion) statement.

The proof of \cref{prop:eliminate-rich-subgaussian} will occupy the remainder of this section. We begin with some preliminary results about the structure of rich vectors.

The first result is a simple observation due to Tao and Vu \cite{tao2008random} showing that for every rich vector, there exists a sufficiently large interval such that the L\'evy concentration function of the vector is `approximately constant' at any radius in this interval. 

\begin{lemma}
\label{lemma:pigeonhole-subgaussian}
For any $\bm{v} \in \bm{R}(\beta)$, there exists some $j\in \{0,1,\dots,J(\beta,n)\}$ such that 
$$\rho_{2\eta \sqrt{n}(2\|M\|f(\beta)^{-1})^{j+1},\xi}(\bm{v}) \leq n^{1/100}\rho_{2\eta \sqrt{n}(2\|M\|f(\beta)^{-1})^{j},\xi}(\bm{v}).$$
\end{lemma}
\begin{remark}
Compared to the trivial covering bound: 
$$\rho_{2\eta \sqrt{n}(2\|M\|f(\beta)^{-1})^{j+1},\xi}(\bm{v}) \lesssim (2\|M\|f(\beta)^{-1})^{2}\rho_{2\eta \sqrt{n}(2\|M\|f(\beta)^{-1})^{j},\xi}(\bm{v}),$$
the above lemma represents a tremendous saving, which will be crucial for our arguments. The factor $n^{1/100}$ in the lemma can be replaced by $n^{1/2 - \epsilon}$ at the expense of choosing different parameters in the rest of this section. 
\end{remark}
\begin{proof}
For $j\in \{0,1,\dots,J(\beta,n)\}$, note that the quantities
$$\rho_{2\eta \sqrt{n}(2\|M\|f(\beta)^{-1})^{j},\xi}(\bm{v})$$
are increasing in $j$, and range between $\beta$ and $1$. Therefore, the pigeonhole principle gives the required conclusion.
\end{proof}
To each $\bm{v}\in \bm{R}(\beta)$, assign such an index $j$ arbitrarily, and denote the set of all vectors in $\bm{R}(\beta)$ indexed $j$ by $\bm{R}_j(\beta)$. This leads to the partition $$\bm{R}(\beta) = \sqcup_{j=0}^{J(\beta,n)}\bm{R}_{j}(\beta).$$ 
We further refine this partition, as in Tao and Vu \cite{tao2008random}.
\begin{definition}
For $j \in \{0,1,\dots,J(\beta,n)\}$ and $\ell\in \{0,1,\dots,\log(\beta^{-1})\}$, we define
$$\bm{R}_{j,\ell}(\beta):=\{\bm{v}\in \bm{R}_j(\beta) \mid \rho_{2\eta \sqrt{n}(2\|M\|f(\beta)^{-1})^{j},\xi}(\bm{v}) \in (2^{-\ell-1},2^{-\ell}]\}.$$
\end{definition}
\noindent In particular, since there are at most $200\log(\beta^{-1})^{2}$ choices of the pair $(j,\ell)$, the following suffices (by the union bound) to prove \cref{prop:eliminate-rich-subgaussian}.

\begin{proposition}
\label{prop:jl-subgaussian}
For any $j\in \{0,1,\dots,J(\beta,n)\}$ and $\ell\in \{0,1,\dots,\log(\beta^{-1})\}$,
$$\Pr\left(\exists \bm{a} \in \bm{R}_{j,\ell}(\beta): \|M_{n}\bm{a}\|_{2}\leq \eta \right) \leq C_{\ref{prop:jl-subgaussian}}\exp(-c_{\ref{prop:jl-subgaussian}}n) ,$$
where $C_{\ref{prop:jl-subgaussian}}\geq 1$ and $c_{\ref{prop:jl-subgaussian}}>0$ are constants depending only on $\xi$.
\end{proposition}

The next structural result, which is an immediate corollary of \cref{prop:refined-diophantine}, shows that every rich vector has a scale at which it can efficiently approximated by a Gaussian integer vector. 
\begin{lemma}
\label{lemma:approximate-jl-subgaussian}
Let $\bm{a}\in \bm{R}_{j,\ell}(\beta)$. Then, there exists some $D \in \C$ with $|D| \in [f(\beta),n^{1/20}]$ and some $\bm{v'}\in (\Z+i\Z)^{n}$ such that
$$\|\bm{v} - \bm{v'}\|_{2} \leq n^{1/20},$$
where $\bm{v}:= (2\eta \sqrt{n})^{-1}(2\|M\|f(\beta)^{-1})^{-j}D\bm{a}$.
\end{lemma}
\begin{proof}
Let 
$$g(n) = n^{1/20}, \quad \bm{w}:= (2\eta \sqrt{n})^{-1}(2\|M\|f(\beta)^{-1})^{-j}\bm{a}.$$ Suppose for contradiction that the desired conclusion does not hold. Then, for all $t\in \C$ such that $|t|\in [f(\beta),g(n)]$,
$$\dist(t\bm{w},\Z^{n}) \geq n^{1/20}. $$
Hence, by \cref{prop:refined-diophantine}, 
\begin{align*}
    \rho_{1,\xi}(\bm{w}) 
    &\leq C_{\ref{prop:refined-diophantine}}\exp(\pi)\left(2\exp(-c_{\ref{prop:refined-diophantine}}n^{1/10}) + f(\beta)\right)\\
    &\leq  3C_{\ref{prop:refined-diophantine}}\exp(\pi)f(\beta) \leq \beta,
\end{align*}
so that
\begin{align*}
   \rho_{2\eta \sqrt{n},\xi}(\bm{a})\leq \rho_{2\eta \sqrt{n}(2\|M\|f(\beta)^{-1})^{j},\xi}(\bm{a}) = \rho_{1,\xi}(\bm{w}) \leq \beta, 
\end{align*}
which contradicts that $\bm{a} \in \bm{R}(\beta)$.
\end{proof}

The utility of the previous lemma is that it allows us to reduce \cref{prop:jl-subgaussian} to a statement about Gaussian integer vectors, which we then prove via a union bound. Indeed, let $\mathcal{O}$ be the event that the operator norm of $N_n$ is at most $C_{\ref{lemma:operator-norm-subgaussian}}\sqrt{n}$. 
By \cref{lemma:operator-norm-subgaussian}, 
$$\Pr\left(\exists \bm{a} \in \bm{R}_{j,\ell}(\beta): \|M_{n}\bm{a}\|_{2}\leq \eta\right) \leq \Pr\left(\{\exists \bm{a} \in \bm{R}_{j,\ell}(\beta): \|M_{n}\bm{a}\|_{2}\leq \eta\}  \cap \mathcal{O}\right) + 2\exp(-n).$$
Suppose the event in the first term on the right occurs. Let $\bm{a} \in \bm{R}_{j,\ell}(\beta)$ be such that $\|M_{n}\bm{a}\|_{2} \leq \eta$, and let $D \in \C$ with $|D|\in [f(\beta),n^{1/20}]$, $\bm{v'}\in (\Z+i\Z)^{n}$ be such that the conclusion of \cref{lemma:approximate-jl-subgaussian} holds for $\bm{a}, D, \bm{v'}$. Let 
$$\bm{v} = (2\eta \sqrt{n})^{-1}(2\|M\|f(\beta)^{-1})^{-j}D\bm{a}.$$ Then, by the triangle inequality, we have
\begin{align*}
    \|M_{n}\bm{v'}\|_{2} 
    &\leq \|M_{n}\bm{v}\|_{2} + \|M_{n}\|\|\bm{v}-\bm{v'}\|_{2}\\
    &\leq (2\eta \sqrt{n})^{-1}(2\|M\|f(\beta)^{-1})^{-j}|D|\eta + \left(\|M\|+C_{\ref{lemma:operator-norm-subgaussian}}\sqrt{n}\right)n^{1/20}\\
    &\leq |D|n^{-1/2} + \left(\|M\|+C_{\ref{lemma:operator-norm-subgaussian}}\sqrt{n}\right)n^{1/20}\\
    &\leq 2\left(\|M\|+C_{\ref{lemma:operator-norm-subgaussian}}\sqrt{n}\right)n^{1/20}\\
    &\leq 3\|M\|n^{1/20},
\end{align*}
where the fourth line holds since $|D|n^{-1/2} \leq n^{1/20}n^{-1/2} \leq 1$, and the last line holds because of the assumption that $\|M\| \geq 2C_{\ref{lemma:operator-norm-subgaussian}}\sqrt{n}$.

Hence, letting $X_i$ denote the $i^{th}$ row of $M_n$, it follows from Markov's inequality that there are at least $n':= n-n^{0.1}$ coordinates $i\in [n]$ for which
$$|X_i \cdot \bm{v'}| \leq 3\|M\|.$$
It follows that
$$\Pr\left(\|M_n \bm{v'}\|_{2} \leq 3\|M\|n^{1/10}\right) \leq \rho_{3\|M\|,\xi}(\bm{v'})^{n-n^{0.1}}.$$
To summarize, setting
\begin{small}
\begin{align*}
    \widetilde{\bm{R}_{j,\ell}}(\beta):&=\\ \{\bm{v'}\in (\Z+i\Z)^{n}& \mid \exists \bm{a}\in \bm{R}_{j,\ell}(\beta), D\in \C \text{ s.t. } |D|\in [f(\beta),n^{1/20}], \|(2\eta \sqrt{n})^{-1}(2\|M\|f(\beta)^{-1})^{-j}D\bm{a}-\bm{v'}\|_{2}\leq n^{1/20}\},
\end{align*}
\end{small}
we have proved
\begin{proposition}
\label{prop:union-bound-subgaussian}
$\Pr\left(\exists \bm{a}\in \bm{R}_{j,\ell}(\beta): \|M_n \bm{a}\|_{2}\leq \eta\right) 
    \leq \sum_{\bm{v'}\in \widetilde{\bm{R}_{j,\ell}}(\beta)}\rho_{3\|M\|,\xi}(\bm{v'})^{n-n^{0.1}} + 2\exp(-n).$ 
\end{proposition}

\subsection{Counting Gaussian integer vectors approximating scaled rich vectors}

In this subsection, we will control the size of $\widetilde{\bm{R}_{j,\ell}}(\beta)$. This is essentially the only place in the argument where we use the subgaussianity of the random variable $\xi$ (via the application of \cref{prop:approximation-sbp-subgaussian}).

\begin{proposition}
\label{prop:counting-subgaussian}
For every $j\in \{0,1,\dots,J(\beta,n)\}$ and $\ell\in \{0,1,\dots,\log(\beta^{-1})\}$,
$$\left|\widetilde{\bm{R}_{j,\ell}}(\beta)\right| \leq C_{\ref{prop:counting-subgaussian}}\left(2^{n^{0.99}} + \left(\frac{64C_{\ref{thm:counting-continuous}}2^{\ell}}{n^{0.10}}\right)^{n}\right),$$
where $C_{\ref{prop:counting-subgaussian}}\geq 1$ is an absolute constant.
\end{proposition}
\begin{remark}
The crucial part of this upper bound is the appearance of a factor of the form $n^{-\epsilon n}$ in the second term.
\end{remark}
\begin{proof}
We will obtain a good lower bound on $\rho_{1,\xi}(\bm{v'})$ and then appeal to \cref{thm:counting-continuous} for a suitable choice of parameters. For the lower bound, let $\bm{v'}\in \widetilde{\bm{R}_{j,\ell}}(\beta)$ and let $\bm{a}\in \bm{R}_{j,\ell}(\beta)$, $D\in \C$ with $|D|\in [f(\beta),n^{1/20}]$ be such that $\|\bm{v}-\bm{v'}\|_{2}\leq n^{1/20}$, where 
$$\bm{v}:= (2\eta \sqrt{n})^{-1}(2\|M\|f(\beta)^{-1})^{-j}D\bm{a}.$$ Then,
\begin{align*}
    \rho_{2n^{0.15},\xi}(\bm{v'})
    &\geq \rho_{n^{0.15},\xi}(\bm{v}) - 3\exp\left(-\frac{c_{\ref{lemma:subgaussian-concentration}}n^{0.30}}{n^{0.10}}\right)\\
    &\geq \rho_{2\eta \sqrt{n}(2\|M\|f(\beta)^{-1})^{j}|D|^{-1}n^{0.15},\xi}(\bm{a}) - 3\exp\left(-c_{\ref{lemma:subgaussian-concentration}}n^{0.20}\right)\\
    &\geq \rho_{2\eta \sqrt{n}(2\|M\|f(\beta)^{-1})^{j}, \xi}(\bm{a}) - 3\exp\left(-c_{\ref{lemma:subgaussian-concentration}}n^{0.20}\right)\\
    &\geq \frac{\rho_{2\eta\sqrt{n}(2\|M\|f(\beta)^{-1})^{j},\xi}(\bm{a})}{2},
\end{align*}
where the first inequality follows from \cref{prop:approximation-sbp-subgaussian}, the third inequality follows since $|D|^{-1}n^{0.15} \geq n^{-1/20}n^{0.15}\geq 1$, and the last inequality follows from $\rho_{2\eta \sqrt{n},\xi}(\bm{a}) \geq \beta \gg \exp(-n^{0.1})$.
Hence, by the pigeonhole principle, we must have
$$\rho_{1,\xi}(\bm{v'}) \geq \frac{\rho_{2n^{0.15},\xi}(\bm{v'})}{(4n^{0.15})^{2}} \geq \frac{\rho_{2\eta\sqrt{n}(2\|M\|f(\beta)^{-1})^{j},\xi}(\bm{a})}{32n^{0.30}} \geq \frac{2^{-\ell}}{64n^{0.30}},$$
where the final inequality holds since $\bm{a}\in \bm{R}_{j,\ell}(\beta)$.
To summarize, using notation as in \cref{thm:counting-continuous}, we have shown that 
$$\widetilde{\bm{R}_{j,\ell}}(\beta)\subseteq \bm{V}_{2^{-\ell}/64n^{0.30}}.$$

Applying \cref{thm:counting-continuous} with the parameters $s= n^{0.9}$, $k= n^{0.1}$, and $p = 2^{n^{0.04}}$, we find that
$$|\varphi_p(\bm{V}_{\rho})| \leq \left(5np^{2}\right)^{n^{0.9}} + \left(\frac{C_{\ref{thm:counting-continuous}}\rho^{-1}}{n^{0.4}}\right)^{n},$$
for all $\rho \geq C_{\ref{thm:counting-continuous}}^{-1}2^{-n^{0.04}/4}.$ In particular, since $$2^{-\ell}/64n^{0.30} \geq \beta/64n^{0.30} \gg 2^{-n^{0.01}},$$ it follows that
$$\left|\varphi_{p}\left(\bm{V}_{2^{-\ell}/64n^{0.30}}\right)\right| \lesssim 2^{n^{0.99}} + \left(\frac{64C_{\ref{thm:counting-continuous}}2^{\ell}}{n^{0.10}}\right)^{n}.$$
Finally, since 
$$\|\bm{v'}\|_{\infty} \leq \|\bm{v'}\|_{2} \leq (2\eta \sqrt{n})^{-1}|D| + n^{1/4} \ll 2^{n^{0.04}},$$
we see that the map $\varphi_p$ is an injection on $\widetilde{\bm{R}_{j,\ell}}(\beta)\subseteq \bm{V}_{2^{-\ell}/64n^{0.30}}$, which completes the proof. 
\end{proof}
\subsection{Proof of \cref{prop:jl-subgaussian}}
Since we already have control on the size of $\widetilde{\bm{R}_{j,\ell}}(\beta)$, in order to prove \cref{prop:jl-subgaussian} via \cref{prop:union-bound-subgaussian}, it suffices to have good control over $\rho_{3\|M\|,\xi}(\bm{v'})$. This is provided by the following lemma. 
\begin{lemma}
\label{lemma:control-sbp-subgaussian}
For any $\bm{v'} \in \widetilde{\bm{R}_{j,\ell}}(\beta)$,
$$\rho_{3\|M\|,\xi}(\bm{v'}) \leq \min\left\{1-\frac{u_{\ref{lemma:invertibility-single-vector}}}{2}, 2n^{1/100}2^{-\ell} \right\}.$$
\end{lemma}
\begin{proof}
Since $4\eta \sqrt{n}(2\|M\|f(\beta)^{-1})^{J(\beta,n)+1} \leq v_{\ref{lemma:invertibility-single-vector}}$, it follows from \cref{prop:approximation-sbp-subgaussian} that (with notation as in the proof of \cref{prop:counting-subgaussian})
\begin{align*}
    \rho_{3\|M\|,\xi}(\bm{v'}) 
    &\leq \rho_{4\|M\|,\xi}(\bm{v}) + 3\exp\left(-\frac{c_{\ref{lemma:subgaussian-concentration}}\|M\|^{2}}{n^{0.10}}\right)\\
    &\leq \rho_{(4\eta \sqrt{n})(2\|M\|f(\beta)^{-1})^{j}(4\|M\||D|^{-1}),\xi}(\bm{a}) + 3\exp\left(-c_{\ref{lemma:subgaussian-concentration}}n^{0.90}\right)\\
    &\leq \rho_{2\eta \sqrt{n}(2\|M\|f(\beta)^{-1})^{J(\beta,n)+1},\xi}(\bm{a}) + 3\exp\left(-c_{\ref{lemma:subgaussian-concentration}}n^{0.90}\right)\\
    &\leq \rho_{v_{\ref{lemma:invertibility-single-vector}},\xi}(\bm{a}) + 3\exp\left(-c_{\ref{lemma:subgaussian-concentration}}n^{0.90}\right)\\
    &\leq 1-\frac{u_{\ref{lemma:invertibility-single-vector}}}{2},
\end{align*}
for all $n$ sufficiently large, where the third line follows from $8\|M\| |D|^{-1} \ll 2\|M\|f(\beta)^{-1}$.
We also have  
\begin{align*}
\rho_{3\|M\|,\xi}(\bm{v'}) &\leq \rho_{4\|M\|,\xi}(\bm{v}) + 3\exp\left(-{c_{\ref{lemma:subgaussian-concentration}}n^{0.90}}\right)\\
&\leq \rho_{(2\eta \sqrt{n})(2\|M\|f(\beta)^{-1})^{j}(4\|M\||D|^{-1}),\xi}(\bm{a}) + 3\exp\left(-c_{\ref{lemma:subgaussian-concentration}}n^{0.90}\right)\\
&\leq \rho_{2\eta \sqrt{n}(2\|M\|f(\beta)^{-1})^{j+1},\xi}(\bm{a}) + 3\exp\left(-c_{\ref{lemma:subgaussian-concentration}}n^{0.90}\right)\\
&\leq n^{1/100}\rho_{2\eta \sqrt{n}(2\|M\|f(\beta)^{-1})^{j},\xi}(\bm{a}) + 3\exp\left(-c_{\ref{lemma:subgaussian-concentration}}n^{0.90}\right)\\
&\leq n^{1/100}2^{-\ell} + 3\exp\left(-c_{\ref{lemma:subgaussian-concentration}}n^{0.90}\right)\\
&\leq 2n^{1/100}2^{-\ell},
\end{align*}
where the fourth line follows from \cref{lemma:pigeonhole-subgaussian}, the fifth line follows since $\bm{a} \in \bm{R}_{j,\ell}(\beta)$, and the last line follows since $2^{-\ell} \geq \beta \gg \exp(-n^{0.2})$.
\end{proof}

The proof of \cref{prop:jl-subgaussian} is now immediate.
\begin{proof}[Proof of \cref{prop:jl-subgaussian}]
We have
\begin{align*}
   \Pr\left(\exists \bm{a}\in \bm{R}_{j,\ell}(\beta): \|M_n \bm{a}\|_{2}\leq \eta\right) 
    &\leq \sum_{\bm{v'}\in \widetilde{\bm{R}_{j,\ell}}(\beta)}\rho_{3\|M\|,\xi}(\bm{v'})^{n-n^{0.10}} + 2\exp(-n)\\
    &\leq |\widetilde{\bm{R}_{j,\ell}}(\beta)|\left(\min\left\{1-\frac{u_{\ref{lemma:invertibility-single-vector}}}{2}, 2n^{1/100}2^{-\ell} \right\}\right)^{n-n^{0.10}} + 2\exp(-n)\\
    &\leq C_{\ref{prop:counting-subgaussian}}\left(2^{n^{0.99}} + \left(\frac{64C_{\ref{thm:counting-continuous}}2^{\ell}}{n^{0.10}}\right)^{n}\right)\left(\min\left\{1-\frac{u_{\ref{lemma:invertibility-single-vector}}}{2}, 2n^{1/100}2^{-\ell} \right\}\right)^{n-n^{0.10}}\\
    & \hspace{0.5cm}+ 2\exp(-n)\\
    &\leq O(\exp(-\Omega({n}))),
\end{align*}
where the first line follows from \cref{prop:union-bound-subgaussian}, the second line follows from \cref{lemma:control-sbp-subgaussian}, the third line follows from \cref{prop:counting-subgaussian}, and the last line follows since $2^{\ell} \leq \beta^{-1} \ll 2^{n^{0.02}}$. 
\end{proof}

\section{Proof of \cref{thm:main-smoothed-analysis}}
\label{sec:proof-main}
\subsection{L\'evy concentration functions of $\ell_{\infty}$-close vectors}
As mentioned earlier, the key technical difficulty in the proof of \cref{thm:main-smoothed-analysis} compared to the proof of \cref{thm:main-subgaussian} is the unavailability of \cref{prop:approximation-sbp-subgaussian}. Instead, we have the following substitute. 

\begin{proposition}
\label{prop:approximation-sbp-heavy}
Let $\bm{\xi}:=(\xi_1,\dots,\xi_n) \in \C^{n}$ be a complex random vector whose entries are independent copies of a complex random variable $\xi$ with mean $0$ and variance $1$. For $\epsilon \in (0,1/2)$, let $\mathcal{G}_{\epsilon}$ denote the event that $\sum_{i=1}^{n}|\xi_i|^{2} \leq n^{1+2\epsilon}$ and $|\sum_{i=1}^{n}\xi_{i}| \leq n^{(1/2) + \epsilon}$.  
%For $t_2,t_\infty \geq 0$, let $\mathcal{G}_{\bm{\xi}, t_2, t_\infty}$ denote the event that $\xi_{1}^{2}+\dots+\xi_{n}^{2} \leq t_2$ and $\sup_{(v_1,\dots,v_n)\in [-1, 1]^{n}}(v_1\xi_1 + \dots + v_n\xi_n) \leq t_{\infty}$. 
Then, for every $\bm{a}:=(a_1,\dots,a_n), \bm{b}:=(b_1,\dots,b_n)\in \C^{n}$, and for every $r_1,t \geq 0$, we have
$$\rho_{r_1+r_2, \bm{\xi}\mid \mathcal{G}_{\epsilon}}(\bm{b}) \geq \rho_{r_1, \bm{\xi}|\mathcal{G}_{\epsilon}}(\bm{a}) - 4\exp\left(-\frac{r_2^{2}}{256n^{1+2\epsilon}\|\bm{a}-\bm{b}\|_{\infty}^{2}}\right),$$
where $r_2:= 2t\|\bm{a}-\bm{b}\|_{\infty}$.
\end{proposition}

In order to prove this proposition, we will need some facts about concentration on the symmetric group. The following appears as Lemma 3.9 in \cite{rebrova2018coverings}, and is a direct application of Theorem 7.8 in \cite{milman2009asymptotic}. 
\begin{lemma}[Lemma 3.9 in \cite{rebrova2018coverings}]
\label{lemma:talagrand-corollary}
Let $\bm{y}:=(y_{1},\dots,y_{n})$ be a non-zero complex
vector and let $\bm{v}\in[-1,1]^{n}$. Consider the function
$h:S_{n}\to\C$ defined by 
\[
h(\pi):=\sum_{j=1}^{n}v_{\pi(j)}y_{j}.
\]
Then, for all $t>0$,
\[
\Pr\left(\left|h(\pi)-\E h\right|\geq t\right)\leq4\exp\left(-\frac{t^{2}}{128\|\bm{y}\|_{2}^{2}}\right),
\]
where the probability is with respect to the uniform measure on $S_n$. 
\end{lemma}
\begin{remark}
In \cite{rebrova2018coverings}, the above lemma is stated for $\bm{v}\in \{\pm 1\}^{n}$, but exactly the same proof shows that the conclusion also holds for any $\bm{v}\in [-1,1]^{n}$. Also, it is stated and proved (with better constants) for real vectors $\bm{y}$. However, the version above for complex vectors immediately follows from this by separately considering the real and imaginary parts of $h$ and using the union bound.
\end{remark}

We will use this lemma via the following immediate corollary.

\begin{lemma}
\label{lemma:concentration-corollary}
Let $\bm{v}:=(v_1,\dots,v_n), \bm{w}:=(w_1,\dots,w_n) \in \C^{n}\setminus \{\bm{0}\}$, and let $\pi$ be a random permutation uniformly distributed on $S_n$.
Consider the function $h:S_n\to \C$ defined by
$$h(\pi):= \sum_{i=1}^{n}v_{\pi(i)}w_{i}.$$
Then, for all $t \geq |w_1+\dots+w_n|$,
$$\Pr\left(|h(\pi)| \geq 2t\|\bm{v}\|_{\infty} \right) \leq 4\exp\left(-\frac{t^{2}}{128\|\bm{w}\|_{2}^{2}}\right).$$
\end{lemma}
 \begin{proof}
% Let $\{a_{ij}\}_{1\leq i,j \leq n}$ be defined by $a_{ij}:= v_i w_j$, and observe that
% $f(\pi) = \sum_{i=1}^{n}a_{i\pi(i)}$. Since $a$ Therefore, from \cref{prop:concentration-stein}, it follows that for any $t\geq 0$,
 First, note that
 \begin{align*}
 \left|\E[h(\pi)]\right| 
 &= \left|\sum_{i=1}^{n}v_i\E\left[w_{\pi(i)}\right]\right| \\
 &= \left|\sum_{i=1}^{n}v_i \frac{(w_{1}+\dots+w_{n})}{n}\right| \\
 &= \left|\frac{(v_1+\dots+v_n)(w_1+\dots+w_n)}{n}\right|\\
 &\leq \|\bm{v}\|_{\infty}|w_1+\dots+w_n|.
 \end{align*}
 Next, let $\bm{v'}:= \|\bm{v}\|_{\infty}^{-1}\bm{v}$. Then, $\bm{v'}\in [-1,1]^{n}$ and $h(\pi) = \|\bm{v}\|_{\infty}g(\pi)$, where $g(\pi):= \sum_{i=1}^{n}{v'}_{\pi(i)}w_i$. 
 Therefore, by \cref{lemma:talagrand-corollary}, for all $t>0$,
 $$\Pr\left(|g(\pi) - \E{g}| \geq t \right) \leq 4\exp\left(-\frac{t^{2}}{128\|\bm{w}\|_{2}^{2}}\right),$$
 so that
 $$\Pr\left(|h(\pi) - \E{h}| \geq t\|\bm{v}\|_{\infty}\right) \leq 4\exp\left(-\frac{t^{2}}{128\|\bm{w}\|_{2}^{2}}\right).$$
 The desired statement now follows from the triangle inequality and the estimate on $\E{h}$. 
 \end{proof}
 
 We can now prove \cref{prop:approximation-sbp-heavy}.

\begin{proof}[Proof of \cref{prop:approximation-sbp-heavy}]
Consider the random variable $X:= \sum_{i=1}^{n}(a_i - b_i)\xi_i$. We claim that for any $t\geq n^{(1/2)+\epsilon}$,
$$\Pr\left(|X|\geq 2t\|\bm{a}-\bm{b}\|_{\infty} \mid \mathcal{G}_{\epsilon}\right) \leq 4\exp\left(-\frac{t^2}{128 n^{1+2\epsilon}}\right).$$
Indeed, since the distribution of the random vector $\bm{\xi}$, even after conditioning on the event $\mathcal{G}_{\epsilon}$, is invariant under permuting its coordinates, it suffices to show (by the law of total probability) that for any \emph{fixed} vector $\bm{w}:=(w_1,\dots,w_n)$ such that $\sum_{i=1}^{n}|w_i|^{2} \leq n^{1+2\epsilon}$ and $|\sum_{i=1}^{n}w_i| \leq n^{(1/2)+\epsilon}$, and for any $t\geq n^{(1/2)+\epsilon}$ 
$$\Pr_{\pi \tilde S_n}\left(\left|\sum_{i=1}^{n}(a_i-b_i)w_{\pi(i)}\right| \geq 2t\|\bm{a}-\bm{b}\|_{\infty}\right) \leq 2\exp\left(-\frac{t^2}{64 n^{1+2\epsilon}}\right).$$
Since $\sum_{i=1}^{n}(a_i - b_i) w_{\pi(i)}$ has the same distribution as $\sum_{i=1}^{n}(a - b)_{\pi(i)}w_i$, this follows immediately from \cref{lemma:concentration-corollary}. 
%We begin by noting that $$\left|\sum_{i=1}^{n}(a_i - b_i)\xi_i\right| \leq t_{\infty}\|\bm{a}-\bm{b}\|_{\infty}.$$

Next, fix $\delta > 0$, and let $x\in \C$ be such that
$$\Pr\left(a_1 \xi_1 + \dots + a_n \xi_n \in B(x,r_1) \mid \mathcal{G}_{\epsilon} \right) \geq \rho_{r_1, \bm{\xi}\mid \mathcal{G}_{\epsilon}}(\bm{a}) - \delta.$$
Then, for any $t \geq n^{(1/2)+\epsilon}$, setting $r_2:= 2t\|\bm{a}-\bm{b}\|_{\infty}$, we have
\begin{align*}
    \Pr\left(b_1 \xi_1 +\dots+b_n\xi_n \in B(x,r_1 + r_2) \mid \mathcal{G}_{\epsilon} \right)
    &\geq \Pr\left(b_1 \xi_1+\dots + b_n \xi_n \in B(x,r_1+r_2) \cap \{|X| \leq r_2\} \mid \mathcal{G}_{\epsilon}\right) \\
    &\geq \Pr\left(a_1\xi_1+\dots +a_n\xi_n \in B(x,r_1) \cap \{|X|\leq r_2\} \mid \mathcal{G}_{\epsilon}\right)\\ 
    &\geq \Pr\left(a_1\xi_1+\dots +a_n\xi_n \in B(x,r_1) \mid \mathcal{G}_{\epsilon}\right) -\Pr\left(|X| \geq r_2 \mid \mathcal{G}_{\epsilon}\right)\\
    &\geq \rho_{r_1, \bm{\xi}|\mathcal{G}_{\epsilon}}(\bm{a})-\delta - 4\exp\left(-\frac{r_2^{2}}{512n^{1+2\epsilon}\|\bm{a}-\bm{b}\|_{\infty}^{2}}\right).
\end{align*}
Taking the supremum of the left hand side over the choice of $x\in \C$, and then taking the limit on the right hand side as $\delta \to 0$ gives the desired conclusion. 
\end{proof}

\subsection{Regularization of $N_n$}
In order to make use of the results of the previous subsection, we need that, with high probability, almost all of the rows of $N_n$ satisfy the event $\mathcal{G}_{\epsilon}$. This follows using a straightforward application of the standard Chernoff bound. 
\begin{lemma}[Lemma 5.3 in \cite{jain2019b}]
\label{lemma:proj_control_basic}
Let $N_n:=(a_{ij})$ be an $n\times n$ complex random matrix with i.i.d. entries,
each with mean $0$ and variance $1$. For $\epsilon\in(0,1/2)$, let $I\subseteq[n]$ denote the (random) subset of coordinates
such that for each $i\in I$,
\begin{equation}
\label{eqn:good-row-condition}
\left(\sum_{j=1}^{n}|a_{ij}|^{2}\leq n^{1+2\epsilon}\right)\bigwedge\left(\left|\sum_{j=1}^{n}a_{ij}\right|\leq n^{(1/2)+\epsilon}\right).
\end{equation}
Let $\mathcal{R}_{\epsilon}$ denote the event that $|I^{c}|\leq 2n^{1-\epsilon}$. 
Then, 
\[
\Pr\left(\mathcal{R}_{\epsilon}^{c}\right)\leq2\exp\left(-\frac{n^{1-\epsilon}}{4}\right).
\]
\end{lemma}

We will also need the following (trival) bound on the probability that the operator norm of $N_n$ is too large. 
\begin{lemma}
\label{lemma:operator-norm}
Let $N_n:=(a_{ij})$ be an $n\times n$ complex random matrix with independent entries, each with mean $0$ and variance $1$. Then, for any $L \geq 1$,
$$\Pr\left(\|N_n\| \geq \sqrt{L}n\right) \leq L^{-1} $$
\end{lemma}
\begin{proof}
By Markov's inequality, $\Pr\left(\sum_{ij}|a_{ij}|^{2} \geq Ln^{2} \right) \leq L^{-1}.$ Since $\|N_n\|^{2} \leq \|N_n\|^{2}_{F} := \sum_{ij}|a_{ij}|^{2}$, the desired conclusion follows.
\end{proof}

Henceforth, let $\mathcal{O}_{\beta}$ denote the event that $\|N_{n}\| \leq \beta^{-1/2}n$; by the above lemma, this occurs except with probability at most $\beta$. Moreover, let $S(\beta):= \|M\| + \beta^{-1/2}n$.
\subsection{Rich and poor vectors}
For the remainder of this section, we fix an $n\times n$ complex matrix $M$ and parameters $\alpha,\eta \in (0,1)$ satisfying the restrictions of the statement of \cref{thm:main-smoothed-analysis}. Also, let $$\beta:= \frac{\alpha}{2n+1},\quad f(\beta):= \frac{\beta}{200C_{\ref{prop:refined-diophantine}}} \in (0,1), \quad J(\beta,n):= \frac{100\log(\beta^{-1})}{\log{n}}, \quad \epsilon = 0.025.$$ We may further assume that $\eta \geq 2^{-n^{0.01}}$, since the statement of \cref{thm:main-smoothed-analysis} for smaller values of $\eta$ follows from the statement for $\eta = 2^{-n^{0.01}}$. 

We call a unit vector $\bm{v}\in \C^{n}$ \emph{poor} if we have
$$\rho_{2\eta S(\beta)\sqrt{n},\bm{\xi}|\mathcal{G}_{\epsilon}}(\bm{v}) \leq \beta$$
and \emph{rich} otherwise. We use $\bm{P}(\beta)$ and $\bm{R}(\beta)$ to denote, respectively, the set of poor and rich vectors. As before, \cref{thm:main-smoothed-analysis} follows from the following two propositions and the union bound. 

\begin{proposition}
\label{prop:eliminate-poor}
$\Pr\left(\exists \bm{v} \in \bm{P}(\beta): \|M_{n}\bm{v}\|_{2} \leq \eta \right) \leq 2n\beta + C_{\ref{prop:eliminate-poor}}\exp(-c_{\ref{prop:eliminate-poor}}n^{2/3}),$ 
where $C_{\ref{prop:eliminate-poor}}\geq 1$ and $c_{\ref{prop:eliminate-poor}}>0$ are constants depending only on $\xi$.
\end{proposition}

\begin{proposition}
\label{prop:eliminate-rich}
$\Pr\left(\exists \bm{v} \in \bm{R}(\beta): \|M_{n}\bm{v}\|_{2} \leq \eta \right) \leq \beta + C_{\ref{prop:eliminate-rich}}\exp(-c_{\ref{prop:eliminate-rich}}n^{2/3})$,
where $C_{\ref{prop:eliminate-rich}}\geq 1$ and $c_{\ref{prop:eliminate-rich}} > 0$ are constants depending only on $\xi$. 
\end{proposition}
\subsection{Eliminating poor vectors}
Compared to \cref{prop:eliminate-poor-subgaussian}, the proof of \cref{prop:eliminate-poor} requires more work, since we need to work with $\rho_{r,\bm{\xi}|\mathcal{G}_{\epsilon}}(\bm{v})$ instead of $\rho_{r,\bm{\xi}}(\bm{v})$. In order to do this, we start by first eliminating `compressible' vectors.   
\begin{definition}[Definition 3.2 in \cite{rudelson2008littlewood}]
Let $\delta_1 \in [0,n], \delta_2 \in (0,1/2)$. 

(1) A vector $\bm{x}\in \C^{n}$ 
is called \emph{sparse} if $|\supp(\bm{x})| \leq \delta_1$. 

(2) A vector  $\bm{x} \in \S^{2n-1}$
is called
\emph{compressible} if $\bm{x}$ is within Euclidean distance $\delta_2$ from the set of all sparse vectors. 

(3) A vector $x\in \S^{2n-1}$ is called \emph{incompressible} if it is not compressible.

The sparse, compressible and incompressible vectors will be denoted respectively by
$\textbf{Sparse}(\delta_1)$, $\textbf{Comp}(\delta_1,\delta_2)$, and $\textbf{Incomp}(\delta_1, \delta_2)$.
\end{definition}

\begin{remark}
\label{rmk:incompressible}
In particular, note that for any $x\in \textbf{Incomp}(\delta_1,\delta_2)$ and for any $I \subseteq [n]$ with $|I|\leq \delta_1 n$, there exists some $j \in I^{c}$ such that $|x_j| \geq \delta_2/\sqrt{n}$. 
\end{remark}

\begin{remark}
We have used the terminology of `compressible' and `incompressible' vectors mostly for convenience, and our use of these notions is rather different from that in the work of Rudelson and Vershynin. In particular, the only property of incompressible vectors we use is captured in the above remark, which is much weaker than what is used by the geometric methods.
\end{remark}

\begin{lemma}
\label{lemma:invertibility-compressible}
Let $\mathcal{C}_{\epsilon,\beta}$ denote the event that there exists some $\bm{v}\in \textbf{Comp}(2n^{1-\epsilon}, S(\beta)^{-1})$ for which $\|\bm{v}^{T}M_{n}\| \leq \eta$. Then,
$$\Pr\left(\mathcal{C}_{\epsilon,\beta}\right) \leq \beta + C_{\ref{lemma:invertibility-compressible}}\exp(-c_{\ref{lemma:invertibility-compressible}}n),$$
where $C_{\ref{lemma:invertibility-compressible}}\geq 1$ and $c_{\ref{lemma:invertibility-compressible}}>0$ are constants depending only on $\xi$.
\end{lemma}

\begin{proof}
By losing an additive error term which is at most $\beta$, it suffices to bound $\Pr(\mathcal{C}_{\epsilon,\beta}\cap \mathcal{O}_\beta)$. Let $\bm{N}$ denote an $S(\beta)^{-1}$-net of $\textbf{Sparse}(2n^{1-\epsilon})\cap \S^{2n-1}$ of minimum cardinality; by a standard volumetric argument, $$|\bm{N}| \leq \binom{n}{2n^{1-\epsilon}}(100S(\beta))^{4n^{1-\epsilon}}.$$

Suppose that $\mathcal{C}_{\epsilon,\beta} \cap \mathcal{O}_\beta$ occurs. Then, by the definition of $\textbf{Comp}(2n^{1-\epsilon},S(\beta)^{-1})$, there exist $\bm{v},\bm{v'} \in \S^{2n-1}$ such that $\|\bm{v}^{T}M_{n}\|_{2} \leq \eta$, $\|\bm{v}-\bm{v'}\|_{2}\leq S(\beta)^{-1}$, and $\bm{v'}$ is supported on at most $2n^{1-\epsilon}$ coordinates.  Moreover, by the definition of $\bm{N}$, there exists some $\bm{v''}\in \bm{N}$ such that $\|\bm{v''}-\bm{v'}\|_{2}\leq S(\beta)^{-1}$. By the triangle inequality, we see that 
\begin{align*}
    \|\bm{v''}^{T}M_{n}\|_{2}
    &\leq \eta + \|M_{n}^{T}\|\|\bm{v}-\bm{v''}\|_{2}\\
    &= \eta + \|M_{n}\|\|\bm{v}-\bm{v''}\|_{2}\\
    &\leq \eta + 2\|M_{n}\|S(\beta)^{-1}\\
    &\leq 2+\eta.
\end{align*}
On the other hand, by \cref{lemma:invertibility-single-vector}, we see that for any fixed $\bm{v''}\in \S^{2n-1}$,
$$\Pr\left(\|\bm{v''}^{T}M_{n}\|_{2} \leq c_{\ref{lemma:invertibility-single-vector}}\sqrt{n}\right) \leq (1-c_{\ref{lemma:invertibility-single-vector}})^{n}.$$
Therefore, taking the union bound over all $\bm{v''}\in \bm{N}$, it follows that $\mathcal{C}_{\epsilon,\beta}\cap \mathcal{O}_{\beta}$ occurs with probability at most 
$$\binom{n}{2n^{1-\epsilon}}(100S(\beta))^{4n^{1-\epsilon}}(1-c_{\ref{lemma:invertibility-single-vector}})^{n} \leq \exp(-\Omega(n)),$$
where the final inequality follows since $S(\beta)^{n^{1-\epsilon}} = O(\exp(o(n))$. 
\end{proof}

\begin{proof}[Proof of \cref{prop:eliminate-poor}]
By \cref{lemma:proj_control_basic} and \cref{lemma:invertibility-compressible}, after losing an additive error term of $\beta +O(\exp(-n^{1-\epsilon}/4))$, it suffices to bound the probability of the event intersected with $\mathcal{C}_{\epsilon,\beta}^{c}\cap \mathcal{R}_\epsilon$. Moreover, 
since 
$$\mathcal{R}_{\epsilon} = \sqcup_{I\subseteq[n], |I| \geq n-2n^{1-\epsilon}} \mathcal{R}_{\epsilon,I},$$
where $\mathcal{R}_{\epsilon,I}$ denotes the event that the rows of $N_n$ satisfying \cref{eqn:good-row-condition} are exactly those indexed by the subset $I$, it suffices (by the law of total probability) to show that for any $I\subseteq[n], |I|\geq n-2n^{1-\epsilon}$, 
$$\Pr\left(\{\exists \bm{v} \in \bm{P}(\beta): \|M_{n}\bm{v}\|_{2} \leq \eta\} \cap \mathcal{C}^{c}_{\epsilon,\beta} \mid \mathcal{R}_{\epsilon, I}\right) \leq n\beta.$$
For the remainder of the proof, fix such an $I$. By reindexing the coordinates, we may further assume that $I = [|I|]$. 

Since $M_n^{\dagger}$ and $M_{n}$ have the same singular values, it follows that a necessary condition for a matrix $M_n$ to satisfy the above event is that there exists a unit vector $\bm{a'}=(a'_{1},\dots,a'_{n})$
such that $\bm{a'} \in \textbf{Incomp}(2n^{1-\epsilon}, S(\beta)^{-1})$ and $\|\bm{a'}^{T}M_{n}\|_{2}\leq \eta$. To every matrix $M_n$, associate such a vector $\bm{a'}$ arbitrarily (if one exists) and denote it by $\bm{a'}_{M_n}$; this leads to a partition of the space of all matrices with least singular value at most $\eta$. By \cref{rmk:incompressible}, since $|I^{c}| \leq 2n^{1-\epsilon}$, there must exist $i\in I$ such that $|(\bm{a'}_{M_n})_i|\geq S(\beta)^{-1}/\sqrt{n}$. To every $\bm{a'}_{M_n}$, associate such an index $i\in I$ arbitrarily, and denote it by $i({M_n})$.
Then, by taking a union bound over the choice of $i \in I$, it suffices to show the following. %over the $n$ possible pieces, it suffices to show the following.
\begin{align}
\label{eqn:intersected-event}
\Pr\left(\{\exists \bm{v}\in \bm{P}(\beta): \|M_n \bm{v}\|_{2} \leq \eta\} %\bigwedge \|\bm{a'}(M_n)^{T}M_n\| \leq \eta 
\cap i({M_n})= 1  \mid \mathcal{R}_{\epsilon, [|I|]}\right) \leq \beta
\end{align}
%t  we  $\bm{a}=(a_1,\dots,a_n)\in\Gamma_{1}(\eta)$
%satisfying $\|M_{n}\bm{a}\|\leq \eta$. Since $M_{n}^{T}$ has the same
%spectral norm as $M_{n}$, there must exist a unit vector $\bm{a'}=(a'_{1},\dots,a'_{n})$
%such that $\|\bm{a'}^{T}M_{n}\|\leq \eta$. By paying an overall factor of $n$, we may assume that $a'_n$ has the largest absolute value among all the coordinates of $\bm{a'}$. 
To this end, we expose the last $n-1$ rows $X_2,\dots,X_n$ of $M_n$. Note that if there is some $\bm{v}\in\bm{P}(\beta)$ satisfying $\|M_{n}\bm{v}\|_{2}\leq \eta$,
then there must exist a vector $\bm{y}\in \bm{P}(\beta)$, depending only on
the last $n-1$ rows $X_{2},\dots,X_{n}$, such that 
\[
\left(\sum_{i=2}^{n}|X_{i}\cdot \bm{y}|^{2}\right)^{1/2}\leq \eta.
\]
In other words, once we expose the last $n-1$ rows of the matrix, either the matrix cannot be extended to one satisfying the event in \cref{eqn:intersected-event}, or there is some unit vector $\bm{y} \in \bm{P}(\beta)$, which can be chosen after looking only at the last $n-1$ rows, and which satisfies the equation above. For the rest of the proof, we condition on the last $n-1$ rows $X_2,\dots,X_{n}$ (and hence, a choice of  $\bm{y}$).

For any vector $\bm{w'}\in \S^{2n-1}$ with $w'_1 \neq 0$, we can write
\[
X_{1}=\frac{1}{w_{1}'}\left(\bm{u}-\sum_{i=2}^{n}w_{i}'X_{i}\right),
\]
where $\bm{u}:= \bm{w'}^{T}M_n$.
Thus, restricted to the event $\{s_n(M_n) \leq \eta\}\cap \{i({M_n}) = 1\}$, we have
\begin{align*}
\left|X_{1}\cdot \bm{y}\right| & =\inf_{\bm{w'}\in \S^{n-1}, w'_1 \neq 0}\frac{1}{|w_{n}'|}\left|\bm{u}\cdot \bm{y}-\sum_{i=2}^{n}w_{i}'X_{i}\cdot \bm{y}\right|\\
 &\leq  %\inf_{\bm{w'}\in \S^{n-1}, |w'_n| \geq n^{-1/2}}
 \frac{1}{|a_{1}'|}\left(\|{\bm{a'}_{M_n}}^{T}M_{n}\|_{2}\|\bm{y}\|_{2}+\|\bm{a'}_{M_n}\|_{2}\left(\sum_{i=2}^{n}|X_{i}\cdot \bm{y}|^{2}\right)^{1/2}\right)\\
 &\leq S(\beta)\eta \sqrt{n}\left(\|\bm{y}\|_{2} + \|\bm{a'}_{M_n}\|_{2}\right) \leq 2S(\beta)\eta \sqrt{n},
\end{align*}
where the second line is due to the Cauchy-Schwarz inequality and the particular choice $\bm{w'}=\bm{a'}_{M_n}$.
%Note that restricted to the event $\|\bm{a'}(M_n)^{T}M_n\| \leq \eta \bigwedge \|\bm{a'}(M_n)\|_{\infty} = |a'_n|$, the infimum in the last equation is at most 
%$$\frac{\eta}{|w_n'|}\left(\|\bm{y}\| + \|\bm{w'}\|\right) \leq 2\eta\sqrt{n}.$$

Since, conditioned on $\mathcal{R}_{\epsilon, [|I|]}$, the first row of $N_n$ is distributed as $\bm{\xi}|\mathcal{G}_{\epsilon}$, it follows that the probability in \cref{eqn:intersected-event} is bounded by
$$\rho_{2\eta S(\beta)\sqrt{n},\bm{\xi}\mid \mathcal{G}_{\epsilon}}(\bm{y}) \leq \beta, $$
which completes the proof. 
\end{proof}

\subsection{Eliminating rich vectors}
Up to losing an overall additive error term of $\beta$, it suffices to bound $\Pr\left(\{\exists \bm{v}\in \bm{R}(\beta): \|M_{n}\bm{v}\|_{2} \leq \eta\} \cap \mathcal{O}_{\beta}\right)$.
Exactly as in the proof of \cref{thm:main-subgaussian}, we obtain the decomposition
$$\bm{R}(\beta) = \sqcup_{j,\ell}\bm{R}_{j,\ell}(\beta),$$
where $j \in \{0,1,\dots, J(\beta,n)\}, \ell \in \{0,1,\dots,\log(\beta^{-1})\}$, and
$$\bm{R}_{j,\ell}(\beta):= \{\bm{v} \in \bm{R}_j(\beta) \mid \rho_{2\eta S(\beta)\sqrt{n}(2S(\beta)f(\beta)^{-1})^{j}}(\bm{v}) \in (2^{-\ell-1}, 2^{-\ell}] \}.$$
Recall that if $\bm{v} \in \bm{R}_j(\beta)$, then
$$\rho_{2\eta S(\beta) \sqrt{n}(2S(\beta)f(\beta)^{-1})^{j+1},\xi}(\bm{v}) \leq n^{1/100}\rho_{2\eta S(\beta) \sqrt{n}(2S(\beta)f(\beta)^{-1})^{j},\xi}(\bm{v}).$$
Since there are at most $(J(\beta,n)+1)(\log(\beta^{-1})+1)$ choices for the pair $(j,\ell)$, by the union bound, it suffices to prove the following analogue of \cref{prop:jl-subgaussian} in order to prove \cref{prop:eliminate-rich}

\begin{proposition}
\label{prop:jl}
For any $j\in \{0,1,\dots,J(\beta,n)\}$ and $\ell\in \{0,1,\dots,\log(\beta^{-1})\}$,
$$\Pr\left(\{\exists \bm{a} \in \bm{R}_{j,\ell}(\beta): \|M_{n}\bm{a}\|_{2}\leq \eta\} \cap \mathcal{O}_\beta \right) \leq C_{\ref{prop:jl-subgaussian}}\exp(-c_{\ref{prop:jl}}n^{2/3}) ,$$
where $C_{\ref{prop:jl}}\geq 1$ and $c_{\ref{prop:jl}}>0$ are constants depending only on $\xi$. \end{proposition}

We begin with the following analogue of \cref{lemma:approximate-jl-subgaussian}

\begin{lemma}
\label{lemma:approximate-jl}
Let $\bm{a}\in \bm{R}_{j,\ell}(\beta)$. Then, there exists some $D \in \C$ with $|D|\in [f(\beta),n^{1/20}]$ and some $\bm{v'}\in (\Z+i\Z)^{n}$ such that
$$\|\bm{v} - \bm{v'}\|_{2} \leq n^{1/20},$$
where $\bm{v}:= (2\eta S(\beta) \sqrt{n})^{-1}(2 S(\beta)f(\beta)^{-1})^{-j}D\bm{a}$.

\end{lemma}
\begin{proof}
Let $g(n) = n^{1/20}$ and $\bm{w}:= (2\eta S(\beta) \sqrt{n})^{-1}(2 S(\beta) f(\beta)^{-1})^{-j}\bm{a}$. Suppose for contradiction that the desired conclusion does not hold. Then, the same computation as in the proof of \cref{lemma:approximate-jl-subgaussian} shows that 
\begin{align*}
   \rho_{2\eta S(\beta) \sqrt{n},\xi}(\bm{a})\leq \rho_{2\eta S(\beta) \sqrt{n}(2S(\beta)f(\beta)^{-1})^{j},\xi}(\bm{a}) = \rho_{1,\xi}(\bm{w}) \leq \beta/2. 
\end{align*}
Finally, since $\Pr(\mathcal{G}_{\epsilon}) > 1/2$ by Markov's inequality, it follows from \cref{lemma:compare-sbp-condition} that
$$\rho_{2\eta S(\beta) \sqrt{n},\bm{\xi}|\mathcal{G}_\epsilon}(\bm{a}) < 2 \rho_{2\eta S(\beta) \sqrt{n},\bm{\xi}}(\bm{a}) < \beta,$$
which contradicts that $\bm{a}\in \bm{R}(\beta)$.
\end{proof}
\begin{comment}
\begin{proposition}[Proposition 1.1. in \cite{chatterjee2007stein}]
\label{prop:concentration-stein}
Let $\{a_{ij}\}_{1\leq i,j \leq n}$ be a collection of numbers in $[0,1]$. Let $X:=\sum_{i=1}^{n} a_{i \pi(i)}$, where $\pi$ is drawn from the uniform distribution over the set of all permutations of $[n]$. Then, for any $t\geq 0$,
$$\Pr\left(|X-\E[X]| \geq t\right) \leq 2\exp\left(-\frac{t^2}{4\E[X] + 2t}\right).$$
\end{proposition}
\begin{remark}
Let $\{a_{ij}\}_{1\leq i,j\leq n}$ be a collection of numbers in $[-1,1]$. We may write $a_{ij} = a^{+}_{ij} - a^{-}_{ij}$ for $a^{\pm}_{ij} \in [0,1]$. Then, setting $X^{\pm} := \sum_{i=1}^{n} a^{\pm}_{i\pi(i)}$, we have that $X = X^{+} - X^{-}$. Since $E[X] = E[X^+] - E[X^-]$, it follows from \cref{prop:concentration-stein}, the triangle inequality, and the union bound that
\begin{align*}
    \Pr\left(|X-\E[X]| \geq t\right) 
    &\leq \Pr\left(|X^{+} - \E[X^{+}]| \geq t/2\right) + \Pr\left(|X^{-} - \E[X^{-}]| \geq t/2|\right)
\end{align*}
\end{remark}
\end{comment}
Define
\begin{footnotesize}
\begin{align*}
    \widetilde{\bm{R}_{j,\ell}}(\beta)&:= \\
    \{\bm{v'}\in (\Z+i\Z)^{n} &\mid \exists \bm{a}\in \bm{R}_{j,\ell}(\beta), D\in \C \text{ s.t.} |D| \in [f(\beta),n^{1/20}],  \|(2\eta S(\beta) \sqrt{n})^{-1}(2S(\beta)f(\beta)^{-1})^{-j}D\bm{a}-\bm{v'}\|_{2}\leq n^{1/20}\}.
\end{align*}
\end{footnotesize}
Then, the same computation as in the subgaussian case shows that if the event in the statement of \cref{prop:jl} occurs, then there must exist some $\bm{v'} \in \widetilde{\bm{R}_{j,\ell}}(\beta)$ for which
$$\|M_{n}\bm{v'}\|_{2} \leq 3S(\beta)n^{1/20}.$$
Hence, letting $X_i$ denote the $i^{th}$ row of $M_n$, it follows from Markov's inequality that, given any $I\subseteq [n]$ with $|I^{c}| \leq 2n^{1-\epsilon}$, there are at least $n-3n^{1-\epsilon}$ coordinates $i\in I$ for which
$$|X_i \cdot \bm{v'}| \leq 3S(\beta).$$
Thus, we see that for any such $I$,
$$\Pr\left(\{\exists \bm{a} \in \bm{R}_{j,\ell}(\beta): \|M_{n}\bm{a}\|_{2}\leq \eta\} \cap \mathcal{O}_\beta \mid \mathcal{R}_{\epsilon, I}\right) \leq \sum_{\bm{v'}\in \widetilde{\bm{R}_{j,\ell}}(\beta)}\rho_{3S(\beta), \bm{\xi}|\mathcal{G}_{\epsilon}}(\bm{v'})^{n-3n^{1-\epsilon}},$$
so that
\begin{align}
\label{eqn:union-bound}
    \Pr\left(\{\exists \bm{a} \in \bm{R}_{j,\ell}(\beta): \|M_{n}\bm{a}\|_{2}\leq \eta\} \cap \mathcal{O}_\beta \right) \leq \sum_{\bm{v'}\in \widetilde{\bm{R}_{j,\ell}}(\beta)}\rho_{3S(\beta), \bm{\xi}|\mathcal{G}_{\epsilon}}(\bm{v'})^{n-3n^{1-\epsilon}} + 2\exp(-n^{1-\epsilon}/4). 
\end{align}

As in \cref{lemma:control-sbp-subgaussian}, we have
\begin{lemma}
For any $\bm{v'} \in \widetilde{\bm{R}_{j,\ell}}(\beta)$,
$$\rho_{3S(\beta),\bm{\xi}|\mathcal{G}_{\epsilon}}(\bm{v'}) \leq \min\left\{1-\frac{u_{\ref{lemma:invertibility-single-vector}}}{2}, 4n^{1/100}2^{-\ell} \right\}.$$
\end{lemma}
\begin{proof}
Since $4S(\beta)\eta \sqrt{n}(2S(\beta)f(\beta)^{-1})^{J(\beta,n)+1} \leq v_{\ref{lemma:invertibility-single-vector}}$, it follows from \cref{prop:approximation-sbp-heavy} that (with notation as in the proof of \cref{lemma:control-sbp-subgaussian})
\begin{align*}
    \rho_{3S(\beta),\bm{\xi}|\mathcal{G}_{\epsilon}}(\bm{v'}) 
    &\leq \rho_{4S(\beta),\bm{\xi}|\mathcal{G}_{\epsilon}}(\bm{v}) + 4\exp\left(-\frac{S(\beta)^{2}}{256n^{1+2\epsilon}n^{1/10}}\right)\\
    &\leq \rho_{(2\eta S(\beta) \sqrt{n})(2S(\beta)f(\beta)^{-1})^{j}(4S(\beta)|D|^{-1}),\bm{\xi}|\mathcal{G}_{\epsilon}}(\bm{a}) + 4\exp\left(-n^{1/3}/256\right)\\
    &\leq \rho_{4\eta S(\beta) \sqrt{n}(2S(\beta)f(\beta)^{-1})^{J(\beta,n)+1},\bm{\xi}|\mathcal{G}_{\epsilon}}(\bm{a}) + 4\exp\left(-n^{1/3}/256\right)\\
    &\leq \rho_{v_{\ref{lemma:invertibility-single-vector}},\bm{\xi}|\mathcal{G}_{\epsilon}}(\bm{a}) + 4\exp\left(-n^{1/3}/256\right)\\
    &\leq \Pr(\mathcal{G}_{\epsilon})^{-1}\rho_{v_{\ref{lemma:invertibility-single-vector}},\bm{\xi}}(\bm{a}) + 4\exp\left(-n^{1/3}/256\right)\\
    &\leq (1-2n^{-2\epsilon})^{-1}\rho_{v_{\ref{lemma:invertibility-single-vector}},\bm{\xi}}(\bm{a}) + 4\exp\left(-n^{1/3}/256\right)\\
    &\leq 1-\frac{u_{\ref{lemma:invertibility-single-vector}}}{2},
\end{align*}
for all $n$ sufficiently large, where the second line follows from $S(\beta)|D|^{-1} \leq S(\beta)f(\beta)^{-1}$ and the third to last line follows from \cref{lemma:compare-sbp-condition}.
We also have  
\begin{align*}
\rho_{3S(\beta),\bm{\xi}|\mathcal{G}_{\epsilon}}(\bm{v'}) 
&\leq \rho_{4S(\beta),\bm{\xi}|\mathcal{G}_{\epsilon}}(\bm{v}) + 4\exp\left(-n^{1/3}/256\right)\\
&\leq \rho_{(2\eta S(\beta) \sqrt{n})(2S(\beta)f(\beta)^{-1})^{j}(4S(\beta)|D|^{-1}), \bm{\xi}|\mathcal{G}_{\epsilon}}(\bm{a}) + 4\left(-n^{1/3}/256\right)\\
&\leq \rho_{4\eta S(\beta) \sqrt{n}(2S(\beta)f(\beta)^{-1})^{j+1}, \bm{\xi}|\mathcal{G}_{\epsilon}}(\bm{a}) + 4\exp\left(-n^{1/3}/256\right)\\
&\leq 2n^{1/100}\rho_{2\eta S(\beta) \sqrt{n}(2S(\beta)f(\beta)^{-1})^{j}, \bm{\xi}|\mathcal{G}_{\epsilon}}(\bm{a}) + 4\exp\left(-n^{1/3}/256\right)\\
&\leq 2n^{1/100}2^{-\ell} + 4\exp\left(-n^{1/3}/256\right)\\
&\leq 4n^{1/100}2^{-\ell},
\end{align*}
which completes the proof.
\end{proof}

Given the previous lemma and \cref{eqn:union-bound}, the same calculation as in the proof of \cref{prop:eliminate-rich-subgaussian} shows that the following suffices to prove \cref{prop:jl}. 
\begin{proposition}
\label{prop:counting}
For every $j\in \{0,1,\dots,J(\beta,n)\}$ and $\ell\in \{0,1,\dots,\log(\beta^{-1})\}$,
$$\left|\widetilde{\bm{R}_{j,\ell}}(\beta)\right| \leq C_{\ref{prop:counting}}\left(2^{2n^{0.99}} + \left(\frac{128C_{\ref{thm:counting-continuous}}2^{\ell}}{n^{0.10}}\right)^{n}\right),$$
where $C_{\ref{prop:counting}}\geq 1$ is an absolute constant.
\end{proposition}

\subsection{Proof of \cref{prop:counting}}
\begin{proof}[Proof of \cref{prop:counting}]
Let $\bm{v'}\in \widetilde{\bm{R}_{j,\ell}}(\beta)$ and let $\bm{a}\in \bm{R}_{j,\ell}(\beta)$, $D \in \C$ with $|D|\in [f(\beta),n^{1/20}]$ be such that $\|\bm{v}-\bm{v'}\|_{2}\leq n^{1/20}$, where $\bm{v}:= (2S(\beta)\eta \sqrt{n})^{-1}(2S(\beta)f(\beta)^{-1})^{-j}D\bm{a}$. Then, there must exist a subset $T\subseteq [n]$ with $|T^{c}| \leq n^{0.95}$ such that $|v_t - v'_t| \leq n^{-0.4}$ for all $t\in T$. 

Let $\bm{v''}$ be the vector which agrees with $\bm{v'}$ on $T$ and with $\bm{v}$ on $T^{c}$. Then, $\|\bm{v''} - \bm{v}\|_{\infty} \leq n^{-0.4}$ so that
\begin{align*}
    \rho_{2n^{0.15},\bm{\xi}|\mathcal{G}_{\epsilon}}(\bm{v''})
    &\geq \rho_{n^{0.15},\bm{\xi}|\mathcal{G}_{\epsilon}}(\bm{v}) - 4\exp\left(-\frac{n^{0.30}}{256n^{1+2\epsilon}\cdot n^{-0.8}}\right)\\
    &\geq \rho_{2\eta S(\beta) \sqrt{n}(2S(\beta)f(\beta)^{-1})^{j}|D|^{-1}n^{0.15},\bm{\xi}|\mathcal{G}_{\epsilon}}(\bm{a}) - 4\exp\left(-n^{0.10 - 2\epsilon}/256\right)\\
    &\geq \rho_{2\eta S(\beta) \sqrt{n}(2S(\beta)f(\beta)^{-1})^{j}, \bm{\xi}|\mathcal{G}_{\epsilon}}(\bm{a}) - 4\exp\left(-n^{0.10-2\epsilon}/256\right)\\
    &\geq \frac{\rho_{2\eta S(\beta)\sqrt{n}(2S(\beta)f(\beta)^{-1})^{j},\bm{\xi}|\mathcal{G}_{\epsilon}}(\bm{a})}{2},
\end{align*}
where the first inequality follows from \cref{prop:approximation-sbp-heavy}, the third inequality follows since $|D|^{-1}n^{0.15} \geq n^{-1/20}n^{0.15}\geq 1$, and the last inequality follows from $\rho_{2\eta S(\beta) \sqrt{n},\xi}(\bm{a}) \geq \beta \gg \exp(-n^{0.01})$.
Hence, by the pigeonhole principle and by \cref{lemma:compare-sbp-condition}, we must have
$$\rho_{1,\xi}(\bm{v''}) \geq \frac{\rho_{1,\bm{\xi}|\mathcal{G}_{\epsilon}}(\bm{v''})}{2} \geq \frac{\rho_{2n^{0.15},\bm{\xi}|\mathcal{G}_{\epsilon}}(\bm{v''})}{64n^{0.30}} \geq \frac{2^{-\ell}}{128n^{0.30}},$$
where the final inequality holds since $\bm{a}\in \bm{R}_{j,\ell}(\beta)$. Let $\bm{v'''}$ denote the integer vector which agrees with $\bm{v''}$ (and hence, $\bm{v'}$) on $T$ and is $0$ on $T^{c}$. Then,
$$\rho_{1,\xi}(\bm{v'''}) \geq \rho_{1,\xi}(\bm{v''}) \geq \frac{2^{-\ell}}{128n^{0.30}}.$$

To summarize, using notation as in \cref{thm:counting-continuous}, we have shown that for every vector
$\bm{v'} \in \widetilde{\bm{R}_{j,\ell}}(\beta)$, there exists some $T\subseteq[n]$ with $|T^{c}| \leq n^{0.95}$ such that $\bm{v'}$ agrees with some element of $\bm{V}_{2^{-\ell}/128n^{0.30}}$ on $T$. Since each coordinate of $\bm{v''}$ is an integer with absolute value at most $\|\bm{v''}\|_{2} \leq (2\eta S(\beta)\sqrt{n})^{-1}D + n^{1/4} \ll 2^{n^{0.01}}$, it follows that 
\begin{align*}
    \left|\widetilde{\bm{R}_{j,\ell}}(\beta)\right| \leq n\binom{n}{n^{0.95}}\left(2^{n^{0.01}}\right)^{2n^{0.95}}\left|\bm{V}_{2^{-\ell}/128n^{0.30}}\right|.
\end{align*}
Finally, the calculation in the proof of \cref{prop:counting-subgaussian} shows that
$$\left|\bm{V}_{2^{-\ell}/128n^{0.30}}\right| \lesssim 2^{n^{0.99}} + \left(\frac{128C_{\ref{thm:counting-continuous}}^{\ell}}{n^{0.10}}\right)^{n},$$
which, together with the previous equation, completes the proof. 
\end{proof}

\section{Proof of \cref{thm:counting-continuous}}
\label{sec:proof-counting-continuous}

The proof of \cref{thm:counting-continuous} consists of six steps. The first three steps are modelled after the proof of the optimal inverse Littlewood-Offord theorem of Nguyen and Vu \cite{nguyen2011optimal}, whereas the last three steps are modelled after Hal\'asz's proof of his anti-concentration inequality \cite{halasz1977estimates}. \\ 

\noindent \textbf{Step 1: Extracting a large sublevel set. }For each integer $1\leq m\leq M$, where $M:= 2s/k$, we define 
\[
S_{m}:=\left\{\xi\in\C:\sum_{i=1}^{n}\|v_{i}\xi\|_{z}^{2}+|\xi|^{2}\leq m\right\}.
\]
Since 
\[
\int_{\C}\exp\left(-\sum_{i=1}^{n}\|v_{i}\xi\|_{z}^{2}/2-\pi|\xi|^{2}\right)d\xi\lesssim\sum_{1\leq m\leq M}\mu(S_{m})\exp(-m/2)+\exp(-M/2),
\]
it follows from \cref{lemma:initial-fourier-bound} that
\[
\rho_{1,z}(\bm{v})\lesssim\sum_{1\leq m\leq M}\mu(S_{m})\exp(-m/2)+\exp(-M/2).
\]
In particular, since it is assumed that $\rho_{1,z}(\bm{v}) \geq C_{\ref{thm:counting-continuous}}\exp(-s/k) = C_{\ref{thm:counting-continuous}}\exp(-M/2)$, it follows that for sufficiently large $C_{\ref{thm:counting-continuous}}\geq 1$, %\emph{ if $\rho_{1,z}(\bm{v})\geq C\exp(-M(n))$},
%where $C>0$ is a sufficiently large absolute constant, 
\begin{align*}
\rho_{1,z}(\bm{v}) & \lesssim\sum_{1\leq m\leq M}\mu(S_{m})\exp(-m/2)\\
 & =\sum_{1\leq m\leq M}\mu(S_{m})\exp(-m/4)\exp(-m/4)\\
 & \lesssim\sum_{1\leq m\leq M}\mu(S_{m})\exp(-m/4)c_{m},
\end{align*}
where 
\[
c_{m}:=\frac{e^{-m/4}}{\sum_{m=1}^{M}e^{-m/4}}.
\]
Note that in the last line, we have used the fact that $\sum_{m=1}^{\infty}e^{-m/4}=O(1)$.
Therefore, by averaging with respect to the probability measure $\{c_{m}\}_{m=1}^{M}$,
it follows that there must exist some non-zero integer $m_{0}\in[1,M]$
for which 
\[
\mu(S_{m_{0}})\gtrsim\rho_{1,z}(\bm{v})\exp(m_{0}/4).
\]\\
\textbf{Step 2: Eliminating the $z$-norm. }From here on, all implicit constants will be allowed to depend on $C_z$. %we make the following assumption. There exists some constant $C_{z} > 1$ such that 
%\begin{equation}
%\label{eqn:assumption-on-z}
%\Pr(1\leq|z_{1}-z_{2}|\leq C_{z})\geq1/2,
%\end{equation}
%where $z_{1},z_{2}$ are i.i.d. copies of $z$. 
%Moreover, all implicit constants appearing below will be allowed to depend on $C_{z}$. For ease of notation, we will also denote $\rho_{1,z}(\bm{v})$ simply by $\rho$. \\
Since $S_{m_0}\subset B(0,\sqrt{m_0})$, it follows (by averaging) that there must exist some
$B(x,1/16C_{z})\subset B(0,\sqrt{m_0})$ for which 
\[
\mu(S_{m_0}\cap B(x,1/16C_{z}))\gtrsim\rho\exp(m_0/4)m_0^{-1} \gtrsim\rho\exp(m_0/8).
\]
Moreover, for $\xi_{1},\xi_{2}\in B(x,1/16C_{z})\cap S_{m_0}$, we have
that 
\begin{itemize}
    \item $\xi_{1}-\xi_{2}\in B(0,1/8C_{z})$, and 
    \item 
$\sum_{i=1}^{n}\|v_i(\xi_{1}-\xi_{2})\|_{z}^{2}  \leq\sum_{i=1}^{n}\left(\|v_i\xi_{1}\|_{z}+\|v_i\xi_{2}\|_{z}\right)^{2}
  \leq2\sum_{i=1}^{n}\left(\|v_i\xi_{1}\|_{z}^{2}+\|v_i\xi_{2}\|_{z}^{2}\right)
  \leq4m_0.$
\end{itemize}
Since for any $A\subseteq \C$, $\mu(A-A)\geq\mu(A)$, it follows that setting 
\[
T_{m_0}:=\left\{\xi\in B(0,1/8C_{z}):\sum_{i=1}^{n}\|v_i\xi\|_{z}^{2}\leq4m_0\right\},
\]
we have that 
\[
\mu(T_{m_0})\gtrsim\rho\exp(m_0/8).
\]
Next, let $y:=z_{1}-z_{2}$, where $z_{1},z_{2}$ are i.i.d. copies of $z$.
Since 
\[
\E_{y}\int_{\C}\sum_{i=1}^{n} \|\Re\{v_iy\xi\}\|_{\R/\Z}^{2}\bm{1}_{T_{m_0}}(\xi)d\xi\leq4m_0\mu(T_{m_0}),
\]
it follows that there exists some $y_0 \in \C$ satisfying $C_{z}^{-1}\leq |y_{0}|\leq C_{z}$ such that
\[
\int_{\C}\sum_{i=1}^{n}\|\Re\{v_iy_{0}\xi\}\|_{\R/\Z}^{2}\bm{1}_{T_{m_0}}(\xi)d\xi\leq4m_0\mu(T_{m_0})\Pr\left(C_{z}^{-1}\leq |y| \leq C_z\right)^{-1} \leq 4C_z m_0\mu(T_{m_0}),
\]
where the final inequality follows from the $C_z$-goodness of $z$.
Hence, by Markov's inequality, 
\[
\mu\left(\left\{\xi\in T_{m_0}:\sum_{i=1}^{n}\|\Re\{v_i y_{0}\xi\}\|_{\R/\Z}^{2}\leq8C_zm_0\right\}\right)\geq\frac{\mu(T_{m_0})}{2}\gtrsim\rho\exp(m_0/8).
\]
Since $T_{m_0} \subset B(0,1/8C_z)$, this shows that
$$\mu\left(\left\{\xi\in B(0,1/8C_z):\sum_{i=1}^{n}\|\Re\{v_i y_{0}\xi\}\|_{\R/\Z}^{2}\leq8C_z m_0\right\}\right)\gtrsim\rho\exp(m_0/8).$$
Finally, after replacing $\xi$ by $y_{0}\xi$, and noting that the change of measure factor lies in $[C_z^{-1},C_z]$, it follows that 
\[
T'_{m_0}:=\left\{\xi\in B(0,1/8):\sum_{i=1}^{n}\|\Re\{v_i\xi\}\|_{\R/\Z}^{2}\leq8C_zm_0\right\}
\]
satisfies 
\[
\mu(T'_{m_0})\gtrsim\rho\exp(m_0/8).
\]\\
\noindent \textbf{Step 3: Discretization of $\xi$. }For $p$ a prime as in the statement of the theorem, %Select a sufficiently large
%prime $p$ (we will discuss the size of this prime later), and 
let
\[
B_{0}:=\left\{\frac{r_1}{p} + i \frac{r_2}{p}:r_1,r_2\in\Z,-\frac{p}{8}\leq r_1,r_2\leq \frac{p}{8}\right\},
\]
%Hence, we have divided $B(0,1/8)$ into $p/4$ intervals. 
%Let 
%$$R:=\left\{(r_1,r_2)\in\Z^{2}:-\frac{p}{8}\leq r_1,r_2\leq \frac{p}{8}\right\},$$
and consider the random set $x+B_{0}$, where $x\in[0,1/p] + i[0,1/p]$ is a uniformly distributed random point. %Then, for each $r\in R$, we have $$\Pr\left(x+\frac{r}{p}\in T'_{m_0}\right)\geq\frac{1}{2}.$$
%Therefore, by linearity of expectation,
Then, by linearity of expectation, we have
\[
\E_{x\in[0,1/p]+i[0,1/p]}\left[\left|(x+B_{0})\cap T'_{m_0}\right|\right]\gtrsim \mu(T'_{m_0})p^{2},
\]
so there exists some $x_{0}\in[0,1/p] + i[0,1/p]$ for which 
\[
|(x_{0}+B_{0})\cap T'_{m_0}|\gtrsim \mu(T'_{m_0})p^{2} \gtrsim \rho\exp(m_0/8)p^{2}.
\]
Let us now `recenter' this shifted lattice. Note that for a fixed $\xi_{0}\in(x_{0}+B_{0})\cap T'_{m_0}$,
we have for any $\xi\in(x_{0}+B_{0})\cap T'_{m_0}$ that 
\[
\sum_{i=1}^{n}\|\Re\{v_i(\xi-\xi_{0})\}\|_{\R/\Z}^{2}\leq2\sum_{i=1}^{n}\left(\|\Re\{v_i\xi\}\|_{\R/\Z}^{2}+\|\Re\{v_i\xi_{0}\}\|_{\R/\Z}^{2}\right)\leq32C_zm_0.
\]
Note also that $\xi_{0}-\xi\in B_{1}:=B_{0}-B_{0}=\{(r_1/p) + i(r_2/p):r_1,r_2\in\Z,-p/4\leq r_1,r_2\leq p/4\}$. Hence, for a fixed $\xi_{0}\in(x_{0}+B_{0})\cap T'_{m_0}$, setting
$$P_{m_0}:= \left\{\xi_0 - \xi: \xi\in(x_{0}+B_{0})\cap T'_{m_0}\right\}$$
gives a subset $P_{m_0} \subset B_{1}$ such that 
\[
|P_{m_0}|\gtrsim \rho\exp(m_0/8)p^{2},
\]
and for all $\xi \in P_{m_0}$,
\[
\sum_{i=1}^{n}\|\Re\{v_i\xi\}\|_{\R/\Z}^{2}\leq32C_zm_0.
\]\\

\noindent \textbf{Step 4: Embedding $P_{m_0}$ into $\F_{p}$ and the Hal\'asz trick.
}Let $V:=\supp(\varphi_p(\bm{v})).$ If $|V|< s$, we proceed directly to Step 6. Otherwise, for $I\subseteq V$ such that $|I|\geq s$, we define the sets
$$P'_m(I):=\left\{r:= r_1 + ir_2\in \F_p + i\F_{p}: \sum_{i\in I}\left\|\frac{\Re\{v_ir\}}{p}\right\|_{\R/\Z}^{2}\leq32C_zm \right\},$$
%where $I\subseteq V:=\supp(\varphi_p(\bm{v}))$.
%$$P'_m:=\left\{r\in \F_p: \sum_{i=1}^{n}\left\|\frac{v_ir}{p}\right\|_{\R/\Z}^{2}\leq64m \right\}.$$
Note that since $v_i \in \Z + i\Z$, the map
$$r\mapsto \left\|\frac{\Re\{v_i r\}}{p}\right\|_{\R/\Z}$$
is indeed well-defined as a map from $\F_p + i\F_p$ to $[0,1]$. 
%\[
%P_{m_0}'=\{k\in\Z,k/p\in P_{m_0}\}
%\]
%\emph{viewed as a subset of $\F_{p}$}. In particular, we will consider iterated
%sumsets of $P_{m_0}'$ \textendash{} we should view this addition as
%taking place in $\F_{p}$. 
Note also that, since $P_{m_0} \subset B_{1}$, the size of $P'_{m_0}(I)$ (as a subset of $\F_p + i\F_p$) is atleast the size of $P_{m_0}$ (as a subset of $\frac{1}{p}\cdot (\Z+i\Z)$) i.e. the way we have defined various objects ensures that there are no wrap-around issues. We claim that for all integers $t\geq1$,
\begin{equation}
\label{eqn:iterated-sum}
tP'_{m}(I)\subseteq P'_{t^{2}m}(I).
\end{equation}
Indeed, for $r_{1},\dots,r_{t}\in P_{m}'(I)\subseteq\F_{p}+i\F_p$, we have
%\emph{using the integrality of the $v_i$ that:}
\begin{align*}
\sum_{i\in I}\left\|\Re\left\{v_i\frac{(r_{1}+\dots+r_{t})}{p}\right\}\right\|_{\R/\Z}^{2} & =\sum_{i\in I}\left\|\frac{\Re\{v_ir_{1}\}}{p}+\dots+\frac{\Re\{v_ir_{t}\}}{p}\right\|_{\R/\Z}^{2}\\
 & \leq\sum_{i\in I}\left(\sum_{j=1}^{t}\left\|\frac{\Re\{v_ir_{j}\}}{p}\right\|_{\R/\Z}\right)^{2}\\
 & \leq\sum_{i\in I} t \sum_{j=1}^{t}\left\|\frac{\Re\{v_ir_{j}\}}{p}\right\|_{\R/\Z}^{2}\\
 & \leq t\sum_{j=1}^{t}\sum_{i\in I}\left\|\Re\{v_ir_{j}/p\}\right\|_{\R/\Z}^{2}\\
 & \leq32C_zt^{2}m,
\end{align*}
which gives the desired inclusion.

We now use the  Cauchy-Davenport theorem for $\F_p + i\F_p \simeq \F_p^{2}$ (see, e.g., \cite{eliahou2007some}), which states that every pair of nonempty $A,B\subseteq \F_p+i\F_{p}$ satisfies 
$$|A+B| \geq \min\{p^{2},|A|+|B|-p\}.$$ It follows that for all integers $t\geq 1$, 
$$|tP'_m(I)| \geq \min\{p^{2}, t|P'_m(I)|-tp \}.$$
Hence, by \cref{eqn:iterated-sum}, we have
\begin{equation}
    \label{eqn:CauchyDavenport}
|P'_{t^2 m}(I)| \geq \min\{p^{2},t|P'_m(I)|-tp\}.
\end{equation}
%Let $\bm{v}_I$ denote the $|I|$-dimensional vector obtained by restricting $\bm{v}$ to the coordinates corresponding to $I$. 
We also  claim that $|P'_m(I)| < p^{2}$ as long as $m\leq |I|/500C_z$. Indeed, since the map $\F_p + i\F_{p} \ni r(=r_1 + ir_2) \mapsto \Re\{ar\} = a_1 r_1 - a_2r_2 \in \F_p$ is a $p$-to-$1$ surjection for every non-zero $a:=a_1+ia_2\in \F_p + i\F_{p}$, we have
\begin{align*}
    \sum_{r\in \F_p + i\F_p}\sum_{i\in I}\left\|\Re\{v_i r\}/p\right\|^{2}_{\R/\Z}
    %&= \sum_{i\in \supp(\bm{v}_I)}\sum_{r\in \F_p}\left\|v_i r/p\right\|^{2}_{\R/\Z}\\
    &= |I|p\sum_{r\in \F_p }\|r/p\|_{\R/\Z}^{2}\\
    &\geq |I|p\cdot\sum_{r'=1}^{(p-1)/2}(r'/p)^{2}\\
    & > \frac{|I|\cdot p^{2}}{15}.
\end{align*}
On the other hand, from the definition of $P'_m(I)$, 
$$\sum_{r\in \F_p+i\F_p}\sum_{i\in I}\left\|\Re\{v_i r\}/p\right\|^{2}_{\R/\Z}
     \leq |P'_m(I)|\cdot 32C_z m + \left(p^{2}-|P'_m(I)|\right)\cdot |I|.$$
Comparing these two bounds proves the claim. Combining this claim with \cref{eqn:CauchyDavenport} along with the assumption that $k \geq 1000C_{z}$ shows that
\begin{align*}
    |P'_{M}(I)| 
    &\gtrsim \sqrt{\frac{M}{m_0}}\left(|P'_{m_{0}}(I)|-p\right)\\
    &\gtrsim \sqrt{\frac{M}{m_0}}|P'_{m_{0}}(I)|\\
    &\gtrsim \sqrt{\frac{M}{m_0}}\rho \exp(m_0/8) p^{2}\\
    &\gtrsim \sqrt{M}\rho \exp(m_0/16)p^{2},
\end{align*}
where the second line follows since $|P'_{m_0}(I)|\geq |P'_{m_0}| \gtrsim \rho p^{2} \geq C_{\ref{thm:counting-continuous}}p$ by assumption. 
\begin{remark}
Whereas we have related the size of $P'_m(I)$ to the size of $P'_{t^2 m}(I)$, \cite{nguyen2011optimal} uses a similar computation to deduce information about the size of iterated sumsets of $\{v_1,\dots,v_n\}$. This information is then combined with Freiman-type inverse theorems to provide structural information about $\{v_1,\dots,v_n\}$. Thus, we see that by `dualizing' the argument in \cite{nguyen2011optimal}, one is able to bypass the need for Freiman-type theorems, as far as the counting variant of the inverse Littlewood-Offord problem is concerned.
\end{remark}
%Since $P'_{m}$ increases with $m$, this is true for all pairs $m,t$
%for which $P'_{t^{2}m}\neq\F_{p}$. In either case, we see that (provided
%that $\rho<n^{-1/2}$) that 
%\[
%|P'_{M}|\geq p\rho\sqrt{M}.
%\]

\noindent \textbf{Step 5: Passing to $R_k(\bm{v})$. }%To summarize, we
%have the following: a vector $V_{\beta}\in\Z^{n}\setminus\{0\}\cap[-\beta^{-1},\beta^{-1}]^{n}$
%such that with $M=n^{\delta/2}$
%$$P'_M:=\left\{k\in \F_p: \sum_{i=1}^{n}\left\|\frac{v_ik}{p}\right\|_{\R/\Z}^{2}\leq64M \right\}.$$
%satisfies $|P'_{M}|\geq p\rho\sqrt{M}$. 
Since $\cos(2\pi x)\geq1-20\|x\|_{\R/\Z}^{2}$ for all $x\in \R$,
it follows that 
\[
P'_{M}(I)\subseteq P''_{M}(I):=\left\{r\in\F_{p}+i\F_p:\sum_{i\in I}\cos(2\pi \Re\{v_ir\}/p)\geq |I|-2000C_zM\right\}.
\]
By considering the random variable $r\ni\F_{p}+i\F_p\mapsto\sum_{i\in I}\cos(2\pi \Re\{v_i r\}/p)$,
we have for any $k\in\N$ that
\begin{align}
\label{eqn:moments}
|P''_{M}(I)|(|I|-2000C_zM)^{2k} & \leq\sum_{r\in\F_{p}+i\F_p}\left|\sum_{j\in I}\cos(2\pi \Re\{v_jr\}/p)\right|^{2k} \nonumber \\
 & =\frac{1}{2^{2k}}\sum_{r\in\F_{p}+i\F_{p}}\left(\sum_{j\in I}e^{2\pi i\Re\{v_jr\}/p}+e^{-2\pi i\Re\{v_jr\}/p}\right)^{2k} \nonumber \\
 & =\frac{1}{2^{2k}}\sum_{r\in\F_{p}+i\F_p}\sum_{\epsilon_{1},\dots,\epsilon_{2k}\in\{\pm1\}}\sum_{j_{1},\dots,j_{2k}\in I}e^{2\pi i\Re\{(\epsilon_{1}v_{j_1}+\dots+\epsilon_{{2k}}v_{j_{2k}})r\}/p} \nonumber \\
 = \frac{1}{2^{2k}}\sum_{r_1\in\F_{p}}\sum_{r_2\in \F_p}\sum_{\epsilon_{1},\dots,\epsilon_{2k}\in\{\pm1\}}\sum_{j_{1},\dots,j_{2k}\in I}&e^{2\pi i(\epsilon_{1}\Re\{v_{j_1}\}+\dots+\epsilon_{{2k}}\Re\{v_{j_{2k}}\})r_1/p}e^{-2\pi i(\epsilon_{1}\Im\{v_{j_1}\}+\dots+\epsilon_{{2k}}\Im\{v_{j_{2k}}\})r_2/p} \nonumber\\
=\frac{1}{2^{2k}}\sum_{\epsilon_{1},\dots,\epsilon_{2k}\in\{\pm1\}}\sum_{j_{1},\dots,j_{2k}\in I}p^{2}\cdot\delta_{0}(\epsilon_{1}&\Re\{v_{j_1}\}+\dots+\epsilon_{2k}\Re\{v_{j_{2k}}\})\cdot \delta_{0}(\epsilon_{1}\Im\{v_{j_1}\}+\dots+\epsilon_{2k}\Im\{v_{j_{2k}}\}) \nonumber \\
 & =\frac{1}{2^{2k}}\sum_{\epsilon_{1},\dots,\epsilon_{2k}\in\{\pm1\}}\sum_{j_{1},\dots,j_{2k}\in I}p^{2}\cdot\delta_{0}(\epsilon_{1}v_{j_1}+\dots+\epsilon_{2k}v_{j_{2k}}),
\end{align}
where the second last line follows again using the integrality of $\Re\{v_{1}\}, \Im\{{v_1}\}\dots,\Re\{v_{n}\},\Im\{v_n\}$.

From here on, we will use a slight modification of the results of \cite{FJLS2018} to finish the proof. We begin with the following key definition.
\begin{definition}
Suppose that $\bm{v}\in (\F_{p}+i\F_{p})^{n}$ for an integer $n$ and a prime $p$, and let $k\in \N$. For every $\alpha \in [-1,1]$, we define $R_k^{\alpha}(\bm{v})$ to be the number of solutions to
$$\pm v_{i_1}\pm \dots \pm v_{i_{2k}} = 0 $$
that satisfy $|\{i_1,\dots,i_{2k}\}|\geq (1+\alpha)k$.
\end{definition}
The following elementary lemma from \cite{FJLS2018} shows that for `small' positive $\alpha$, $R_k^{\alpha}(\bm{v})$ is not much smaller than $R_k^{-1}(\bm{v})$.
\begin{lemma}[Lemma 1.6 in \cite{FJLS2018}]
For all integers $k,n$ with $k\leq n/2$, any prime $p$, vector $\bm{v}\in (\F_{p}+i\F_{p})^{n}$, and $\alpha \in [0,1]$,
$$R_k^{-1}(\bm{v}) \leq R_k^{\alpha}(\bm{v}) + (40k^{1-\alpha}n^{1+\alpha})^{k}.$$
\end{lemma}
\begin{proof}
  By definition, $R_k^{-1}(\bm{v})$ is equal to $R_k^{\alpha}(\bm{v})$ plus the number of solutions to $\pm v_{i_1}\pm v_{i_2}\dotsb\pm v_{i_{2k}} = 0$ that satisfy $|\{i_1, \dotsc, i_{2k}\}| < (1+\alpha)k$. The latter quantity is bounded from above by the number of sequences $(i_1, \dotsc, i_{2k}) \in [n]^{2k}$ with at most $(1+\alpha)k$ distinct entries times $2^{2k}$, the number of choices for the $\pm$ signs. Thus
  \[
    R_k^{-1}(\bm{v}) \leq R_k^{\alpha}(\bm{v}) + \binom{n}{(1+\alpha)k} \big((1+\alpha)k\big)^{2k}2^{2k} \leq R_k^{\alpha}(\bm{v}) +  \left(4e^{1+\alpha}k^{1-\alpha}n^{1+\alpha}\right)^k,
  \]
  where the final inequality follows from the well-known bound $\binom{a}{b} \le (ea/b)^b$. Finally, noting that $4e^{1+\alpha} \leq 4e^{2} \leq 40$ completes the proof.
\end{proof}
Let $\bm{v}_I$ denote the $|I|$-dimensional vector obtained by restricting $\bm{v}$ to the coordinates corresponding to $I$. Recognizing the right hand side of \cref{eqn:moments} as
\[
\frac{p^{2}R_{k}^{-1}(\bm{v}_I)}{2^{2k}},
\]
it follows from \cref{eqn:moments} and the above lemma that for any $k \leq \sqrt{|I|}$ and $\alpha \in [0,1/8]$,
\begin{align*}
R_{k}^{\alpha}(\bm{v}_I) & \gtrsim (|I|-2000C_zM)^{2k}2^{2k}\rho\sqrt{M} - (40k^{1-\alpha}|I|^{1+\alpha})^{k}\\
 & \gtrsim |I|^{2k}2^{2k}\rho\sqrt{M} - (40k^{1-\alpha}|I|^{1+\alpha})^{k}\\
 &\gtrsim |I|^{2k}2^{2k}\rho\sqrt{M} - (40 |I|^{(3/2)+\alpha})^{k}\\
 &\gtrsim |I|^{(3/2)k}\left(2^{2k}\sqrt{|I|}^{k}\rho \sqrt{M} - (40)^{k}|I|^{\alpha k}\right)\\
 & \gtrsim |I|^{(3/2)k}\left(2^{2k}\sqrt{|I|}^{k}\rho \sqrt{M}\right)\\
 &\gtrsim |I|^{2k}2^{2k}\rho \sqrt{M},
\end{align*}
where the second line follows from the assumption that $Mk\leq 2s \leq 2|I|$, the third line follows from the assumption that $k\leq \sqrt{s} \leq \sqrt{|I|}$, and the fifth line follows from the assumption that $\rho > s^{-k/4} \geq s^{-(k/2)+2\alpha k}\geq  |I|^{-(k/2) + 2\alpha k}$. \\

\noindent\textbf{Step 6: Applying the counting lemma. } Let us summarize where we stand. We have proved that for any complex random variable $z$ satisfying \cref{eqn:assumption-on-z}, there exists an absolute constant $C:=C(C_z)\geq 1$ for which the following holds. If $\bm{v} \in (\Z+i\Z)^{n}$ satisfies $\rho_{1,z}(\bm{v}) := \rho \geq C_{\ref{thm:counting-continuous}}\max\{e^{-s/k}, s^{-k/4}\}$  for some $1000C_{z} \leq k\leq \sqrt{s} \leq s \leq n/\log{n}$ and sufficiently large $C_{\ref{thm:counting-continuous}}$, and if $\alpha \in [0,1/8]$, then either
\begin{enumerate}
    \item $|V|< s$ (where $V:=\supp(\varphi_p(\bm{v}))$), or
    \item for all $I\subseteq V$ with $|I|\geq s$, 
\end{enumerate}
$$R_k^{\alpha}(\bm{v}_I) \geq \frac{|I|^{2k}2^{2k}\rho \sqrt{M}}{C}. $$
Hence, it follows that
\begin{equation}
\label{eqn:size-image-decompose}
\varphi_{p}\left(\bm{V}_\rho\right)\subseteq \bm{X}_s +   \bigcup_{m=s}^{n}\bm{Y}_{k,s,\rho}^{\alpha}(m),    
\end{equation}
where
$$ \bm{X}_s := \left\{\bm{a}\in (\F_p+i\F_p)^{n}: |\supp(\bm{a})| < s\right\} ,$$
and
\begin{small}
$$\bm{Y}^{\alpha}_{k,s,\rho}(m):= \left\{\bm{a} \in (\F_{p}+i\F_{p})^{n} : |\supp(\bm{a})|= m \text{ and } R^\alpha_k(\bm{a}_I)\geq \frac{2^{2k} |I|^{2k}\rho \sqrt{M}}{C} \forall I\subseteq \supp(\bm{a}) \text{ with }|I|\geq s\right\}.$$
\end{small}
We will bound the size of each of these pieces separately. For $|\bm{X}_s|$, the following simple bound suffices:
\begin{equation}
\label{eqn:size-of-x}
|\bm{X}_{s}|  \leq\sum_{\ell=0}^{s-1}{n \choose \ell}(p^{2})^{\ell}
  \leq s{n \choose s}p^{2s}
  \leq s\left(\frac{enp^{2}}{s}\right)^{s} \leq \left(\frac{5np^{2}}{s}\right)^{s}.
\end{equation}
On the other hand, the desired bound on $\bm{Y}^{\alpha}_{k,s,\rho}(m)$ follows easily from a slight modification of the work in \cite{FJLS2018}. 
\begin{theorem}
\label{thm:counting-lemma}
 Let $p$ be a prime, let $k,n\in \N$, $s\in [n]$, $t\in [p]$, and let $\alpha \in (0,1)$. Denoting
 $$\Bad^{\alpha}_{k,s,\geq t}(n):= \left\{\bm{v} \in (\F_{p}+i\F_p)^{n} : R^\alpha_k(\bm{v}_I)\geq t\cdot \frac{2^{2k} \cdot |I|^{2k}}{p} \text{ for every } I\subseteq [n] \text{ with }|I|\geq s\right\},$$
  we have
  \[
    |\Bad_{k,s,\geq t}^{\alpha}(n)| \leq (\alpha t)^{s-n} p^{n+s}.
  \]
\end{theorem}
The proof of this theorem follows easily from a slight modification of the proof of Theorem 1.7 in \cite{FJLS2018}. For the reader's convenience, we provide complete details in \cref{sec:appendix-counting-lemma}. 
 \begin{corollary}
 \label{corollary:size-of-y}
For our choice of parameters, $|\bm{Y}_{k,s,\rho}^{\alpha}(m)| \leq \left(\frac{16C}{\rho \sqrt{M}}\right)^{n}.$ 
 \end{corollary}
\begin{proof}
After paying an overall factor of $\binom{n}{m}$, it suffices to count only those $\bm{a}\in \bm{Y}^{\alpha}_{k,s,\rho}(m) $ for which $\supp(\bm{a}) = [m]$. The key point is that, by definition, for any such $\bm{a}$, we have
$$\bm{a}|_{[m]} \in \Bad^{\alpha}_{k,s,\geq t}(m),$$
for $t = \lfloor p\rho \sqrt{M}/C \rfloor$.
Therefore, by \cref{thm:counting-lemma}, it easily follows that
\begin{align*}
|\bm{Y}_{k,s,\rho}^{\alpha}(m)| & \leq\binom{n}{m}\left(\alpha tp\right)^{s}\left(\frac{p}{t}\right)^{m}\\
 & \leq2^{n}(tp)^{s}\left(\frac{p}{t}\right)^{n}\\
 & \leq2^{n}\left(p^{2}\sqrt{M}\right)^{s}\left(\frac{2Cp}{p\rho\sqrt{M}}\right)^{n}\\
 & \leq(p^{2}\sqrt{M})^{s}\left(\frac{4C}{\rho\sqrt{M}}\right)^{n}\\
 & \leq \left(\frac{16C}{\rho\sqrt{M}}\right)^{n},
\end{align*}
as desired. 
\end{proof}
From \cref{eqn:size-image-decompose,eqn:size-of-x,corollary:size-of-y}, and noting that $M=2s/k$, it follows that
\begin{align*}
    |\varphi_p(\bm{V}_\rho)| 
    &\leq \left(\frac{5np^{2}}{s}\right)^{s} + n\cdot \left(\frac{16C\rho^{-1}}{\sqrt{s/k}}\right)^{n}\\
    &\leq \left(\frac{5np^{2}}{s}\right)^{s} +  \left(\frac{32C\rho^{-1}}{\sqrt{s/k}}\right)^{n}\\
    &\leq \left(\frac{5np^{2}}{s}\right)^{s} +  \left(\frac{C_{\ref{thm:counting-continuous}}\rho^{-1}}{\sqrt{s/k}}\right)^{n}, 
\end{align*}
where the final inequality follows since we can take $C_{\ref{thm:counting-continuous}}$ larger than $32C$. This completes the proof of \cref{thm:counting-continuous}.

\appendix
\section*{Appendix}

\section{Proof of \cref{thm:counting-lemma}}
\label{sec:appendix-counting-lemma}
In this section, we prove~\cref{thm:counting-lemma} using an elementary double counting argument appearing in \cite{FJLS2018}.

\begin{proof}
  Let $\mathcal{Z}$ be the set of all triples
  \[
    \left(I, \left(i_{s+1},\dots,i_{n}\right), \left(F_{j},{\bm{\epsilon}}^{j}\right)_{j=s+1}^{n} \right),
  \]
  where
  \begin{enumerate}[{1.}]
  \item $I \subseteq [n]$ and $|I|=s$, 
  \item $(i_{s+1},\dotsc,i_n) \in [n]^{n-s}$ is a permutation of $[n]\setminus I$,
  \item each $F_{j}:=(\ell_{j,1},\dotsc,\ell_{j,2k})$ is a sequence of $2k$ elements of $[n]$, and
  \item $\bm{\epsilon}^j\in \{\pm 1\}^{2k}$ for each $j$,
  \end{enumerate}
  that satisfy the following conditions for each $j$:
  \begin{enumerate}[{a.}]
  \item
    \label{item:Z-condition-1}
    $\ell_{j,2k}= i_{j}$ and
  \item
    \label{item:Z-condition-2}
    $(\ell_{j,1},\dotsc,\ell_{j,2k-1}) \in \big(I\cup \{i_{s+1},\dotsc,i_{j-1}\}\big)^{2k-1}$.
  \end{enumerate}

  \begin{claim}
    The number of triples in $\mathcal{Z}$ is at most $(s/n)^{2k-1} \cdot \big(2^{n-s} n! / s!\big)^{2k}$.
  \end{claim}
  \begin{proof}
    One can construct any such triple as follows. First, choose an $s$-element subset of $[n]$ to serve as $I$. Second, considering all $j \in \{s+1, \dotsc, n\}$ one by one in increasing order, choose: one of  the $n-j+1$ remaining elements of $[n] \setminus I$ to serve as $i_j$; one of the $2^{2k}$ possible sign patterns to serve as $\bm{\epsilon}^{j}$; and one of the $(j-1)^{2k-1}$ sequences of $2k-1$ elements of $I \cup \{i_{s+1},\dots,i_{j-1}\}$ to serve as $(\ell_{j,1},\dotsc,\ell_{j,2k-1})$. Therefore,
    \begin{align*}
      |\mathcal{Z}| & \le \binom{n}{s} \cdot \prod_{j=s+1}^n \left((n-j+1) \cdot 2^{2k} \cdot (j-1)^{2k-1}\right) \\
            & = \frac{n!}{s!(n-s)!} \cdot (n-s)! \cdot 2^{2k(n-s)} \cdot \left(\frac{(n-1)!}{(s-1)!}\right)^{2k-1} = \left(\frac{s}{n}\right)^{2k-1} \cdot \left(2^{n-s} \cdot \frac{n!}{s!}\right)^{2k}.\qedhere
    \end{align*}
  \end{proof}
  
  We call $\bm{a} = (a_1, \dotsc, a_n) \in (\F_p+i\F_p)^n$ \emph{compatible} with a triple from $\mathcal{Z}$ if for every $j \in \{s+1, \dotsc, n\}$,
  \begin{equation}
    \label{eq:compatibility}
    \sum_{i=1}^{2k}\bm\epsilon^{j}_ia_{\ell_{j,i}}=0.
  \end{equation}
  
  \begin{claim}
    Each triple from $\mathcal{Z}$ is compatible with at most $p^{2s}$ sequences $\bm{a} \in (\F_p+i\F_p)^n$.
  \end{claim}
  \begin{proof}
    Using~\ref{item:Z-condition-1}, we may rewrite~\cref{eq:compatibility} as
    \[
      \bm{\epsilon}^{j}_{2k}a_{i_{j}} = -\sum_{i=1}^{2k-1}\bm\epsilon^{j}_ia_{\ell_{j,i}}.
    \]
    It follows from~\ref{item:Z-condition-2} that once a triple from $\mathcal{Z}$ is fixed, the right-hand side above depends only on those coordinates of the vector $\bm{a}$ that are indexed by $i \in I \cup \{i_{s+1}, \dotsc, i_{j-1}\}$. In particular, for each of the $p^{2s}$ possible values of $(a_i)_{i \in I}$, there is exactly one way to extend it to a sequence $\bm{a} \in (\F_p+i\F_p)^n$ that satisfies~\cref{eq:compatibility} for every $j$.
  \end{proof}

  \begin{claim}
    Each sequence $\bm{a} \in\Bad_{k,s, \ge t}^\alpha$ is compatible with at least
    \[
     \left(\frac{2^{n-s}n!}{s!}\right)^{2k} \cdot \left(\frac{\alpha t}{p}\right)^{n-s}
    \]
    triples from $\mathcal{Z}$.
  \end{claim}
  \begin{proof}
    Given any such $\bm{a}$, we may construct a compatible triple from $\mathcal{Z}$ as follows. Considering all $j \in \{n, \dotsc, s+1\}$ one by one in decreasing order, we do the following. First, we find an arbitrary solution to
    \begin{equation}
      \label{eq:a-ell-solution}
      \pm a_{\ell_1} \pm a_{\ell_2} \pm \dotsb \pm a_{\ell_{2k}} = 0
    \end{equation}
    such that $\ell_1, \dotsc, \ell_{2k} \in [n]\setminus \{i_{n},\dots,i_{j+1}\}$ and such that $\ell_{2k}$ is a non-repeated index (i.e., such that $\ell_{2k} \neq \ell_i$ for all $i \in [2k-1]$). Given any such solution, we let $\ell_{2k}$ serve as $i_j$, we let the sequence $(\ell_1, \dotsc, \ell_{2k})$ serve as $F_j$, and we let $\bm{\epsilon}^j$ be the corresponding sequence of signs (so that~\cref{eq:compatibility} holds). The assumption that $\bm{a} \in \Bad_{k,s,\geq t}^\alpha(n)$ guarantees that there are at least  $t \cdot \frac{2^{2k}\cdot (n-j+1)^{2k}}{p}$ many solutions to~\cref{eq:a-ell-solution}, each of which has at least $2\alpha k$ nonrepeated indices. Since the set of all such solutions is closed under every permutation of the $\ell_i$s (and the respective signs), $\ell_{2k}$ is a non-repeated index in at least an $\alpha$-proportion of them. Finally, we let $I = [n] \setminus \{i_n, \dotsc, i_{s+1}\}$. Since different sequences of solutions lead to different triples, it follows that the number $Z$ of compatible triples satisfies
    \[
      Z \ge \prod_{j = s+1}^{n} \left(\alpha t \cdot \frac{2^{2k} \cdot (n-j+1)^{2k}}{p}\right) = \left(\frac{2^{n-s}n!}{s!}\right)^{2k}\cdot \left(\frac{\alpha t}{p}\right)^{n-s}.\qedhere
    \]
  \end{proof}

  Counting the number $P$ of pairs of $\bm{a} \in \Bad_{k, s, \ge t}^\alpha(n)$ and a compatible triple from $\mathcal{Z}$, we have
  \[
    |\Bad_{k, s, \ge t}^\alpha(n)| \cdot \left(\frac{2^{n-s}n!}{s!}\right)^{2k} \cdot \left(\frac{\alpha t}{p}\right)^{n-s}\le P \le |\mathcal{Z}| \cdot p^{2s} \le \left(\frac{s}{n}\right)^{2k-1} \cdot \left(\frac{2^{n-s}n!}{s!}\right)^{2k} \cdot p^{2s},
  \]
  which yields the desired upper bound on $|\Bad_{k, s, \ge t}^\alpha(n)|$.
\end{proof}

\section*{Acknowledgements}
This work was done when the author was a PhD student at the Massachusetts Institute of Technology. The author is grateful to Nick Cook for helpful comments on an earlier version of this paper, including the suggestion to consider the complex setting, to Galyna Livshyts and Konstantin Tikhomirov for discussions about their recent work \cite{livshyts2019smallest}, and to anonymous referees for their careful reading of the manuscript and for helpful comments. 

\bibliographystyle{amsplain}
\bibliography{least-singular-value}

\begin{dajauthors}
\begin{authorinfo}[vj]
  Vishesh Jain\\
  Stanford University\\
  Stanford, CA, USA\\
  visheshj\imageat{}stanford\imagedot{}edu \\
  \url{https://jainvishesh.github.io/}
\end{authorinfo}
\end{dajauthors}

\end{document}